\documentclass[12pt]{article}

\usepackage{a4wide}
\usepackage{amssymb,amsmath,array,amsthm,amsfonts,amscd}
\usepackage{version}
\usepackage{algpseudocode}
\usepackage{tikz}
\usepackage{mathrsfs}

\usepackage{caption}
\usepackage{color}

\allowdisplaybreaks[1]

\newcommand{\titre}{A quadratic Poisson Gel'fand-Kirillov problem in prime characteristic}

\newtheorem{thm}{Theorem}[section]
\newtheorem{lem}[thm]{Lemma}
\newtheorem{prop}[thm]{Proposition}

\theoremstyle{definition}
\newtheorem{Def}[thm]{Definition}

\newtheorem{rems}[thm]{Remarks}
\newtheorem{ex}[thm]{Example}
\newtheorem{hypo}{Hypothesis}[subsection]
\newtheorem*{nota}{Notation}

\newcommand{\Z}{\mathbb{Z}}
\newcommand{\C}{\mathbb{C}}
\newcommand{\K}{\mathbb{K}}
\newcommand{\LL}{\mathbb{L}}

\newcommand{\al}{\alpha}
\newcommand{\be}{\beta}

\newcommand{\gam}{\gamma}
\newcommand{\de}{\delta}
\newcommand{\la}{\lambda}
\newcommand{\bo}{\boldsymbol}
\newcommand{\F}{\mbox{\rm Frac\,}}

\newcommand{\Sp}{\mbox{\rm Spec\,}}
\newcommand{\car}{\mbox{\rm char\,}}
\newcommand{\id}{\mbox{\rm id}}
\newcommand{\D}{\mbox{\rm Der}}
\newcommand{\ov}{\overline}
\newcommand{\ul}{\underline}

\excludeversion{hidden}

\begin{document}

\title{\titre} \author{S Launois\thanks{The first author is grateful for the
financial support of EPSRC first grant \textit{EP/I018549/1}.}~ and C Lecoutre\thanks{The second author thanks EPSRC for its support.}}
%%%%%%\author{} \date{}%%%%%%%%

\date{}

\maketitle

\begin{center}
\textit{School of Mathematics, Statistics and Actuarial Science (SMSAS), Cornwallis Building, University of Kent, Canterbury, Kent CT2 7NF, United Kingdom}
\end{center}
\begin{center}
{\tt S.Launois@kent.ac.uk} and {\tt cl335@kent.ac.uk}
\end{center}

%%%%%%%%%%%%%%%%%%%%%%%%%%%%%%%%%%%%%%%%%%%%%%%%%%%%%%%%%%

\begin{abstract}\footnotesize 
The quadratic Poisson Gel'fand-Kirillov problem asks whether the field of fractions of a Poisson algebra is Poisson birationally equivalent to a Poisson affine space, i.e. to a polynomial algebra $\K[X_1, \dots , X_n]$ with Poisson bracket defined by $\{X_i,X_j\}=\lambda_{ij} X_iX_j$ for some skew-symmetric matrix $(\lambda_{ij}) \in M_n(\K)$. This problem was studied in \cite{GL} over a field of characteristic $0$ by using a Poisson version of the deleting derivation homomorphism of Cauchon. In this paper, we study the quadratic Poisson Gel'fand-Kirillov problem over a field of arbitrary characteristic. In particular, we prove that  the quadratic Poisson Gel'fand-Kirillov problem is satisfied for a large class of Poisson algebras arising as semiclassical limits of quantised coordinate rings. For, we introduce the concept of {\it higher Poisson derivation} which allows us to extend the Poisson version of the deleting derivation homomorphism from the characteristic 0 case to the case of arbitrary characteristic.

When a torus is acting rationally by Poisson automorphisms on a Poisson polynomial algebra arising as the semiclassical limit of a quantised coordinate ring, we prove (under some technical assumptions) that quotients by Poisson prime torus-invariant ideals also satisfy the quadratic Poisson Gel'fand-Kirillov problem. In particular, we show that coordinate rings of determinantal varieties satisfy the quadratic Poisson Gel'fand-Kirillov problem.

\end{abstract}

%%%%%%%%%%%%%%%%%%%%%%%%%%%%%%%%%%%%%%%%%%%%%%%%%%%%%%%%%%



\section{Introduction}

Let $\K$ be a field. Recall that a Poisson $\K$-algebra is a commutative algebra endowed with a Poisson bracket, i.e. a skew-symmetric $\K$-bilinear map from $A\times A$ to $A$ satisfying the Jacobi identity and the Leibniz rule. Assuming that $A$ is a domain, we can uniquely extend the Poisson bracket to the field of fractions $\F A$ of $A$. This article is concerned with the Poisson structure of the field of fractions of Poisson polynomial algebras. Examples of Poisson polynomial algebras include the so-called {\it Poisson-Weyl algebras}. Recall that the Poisson-Weyl algebra of dimension $2k$ is the polynomial algebra in $2k$ generators $X_1,\dots,X_k,Y_1,\dots,Y_k$ endowed with the Poisson bracket defined on the generators by $\{X_i,X_j\}=\{Y_i,Y_j\}=0$ and $\{X_i,Y_j\}=\de_{ij}$ for all $i,j$. The field of fractions of this Poisson algebra is referred to as the {\it Poisson-Weyl field of dimension $2k$}. It is a central object in the theory, and often, for a given Poisson polynomial algebra, one tries to decide whether it is Poisson birationally equivalent to a Poisson-Weyl algebra, that is we would like to know whether there exists a Poisson isomorphism between the field of fractions of the given Poisson polynomial algebra and a Poisson-Weyl field of appropriate dimension.

This problem was first raised by Vergne in \cite{Ver},  where the author studied the case of the symmetric algebra $S(\mathfrak{g})$ of a finite dimensional Lie algebra $\mathfrak{g}$ over a field $\LL$ of characteristic $0$, the polynomial algebra $S(\mathfrak{g})$ being endowed with the so-called {\it Kirillov-Kostant-Souriau Poisson structure}: for a basis $U_1,\dots,U_n$ of $\mathfrak{g}$, the Poisson bracket on $S(\mathfrak{g})$ is given by $\{U_i,U_j\}=[U_i,U_j]_{\mathfrak{g}}$ for all $i,j$.  When $\mathfrak{g}$ is nilpotent, Vergne showed that the field of fractions of $S(\mathfrak{g})$ is Poisson isomorphic to the field of fractions of a Poisson-Weyl algebra over a purely transcendental extension of $\LL$. In \cite{TV}, this result was extended to the solvable case by Tauvel and Yu. Moreover, still assuming $\mathfrak{g}$ is solvable, they proved that this result also holds for any quotient of $S(\mathfrak{g})$ by a Poisson prime ideal.
	
The problem raised by Vergne takes its roots in the celebrated Gel'fand-Kirillov Conjecture \cite{GK} which is a problem of birational equivalence between enveloping algebras of Lie algebras and Weyl skew-fields. This conjecture was first proved to fail in general by Alev-Ooms-Van den Bergh \cite{contrex}. See \cite[1]{Pr} for a survey of the results concerning this conjecture. Note that the algebras involved are considered over algebraically closed fields of characteristic zero. However, the conjecture also makes sense in positive characteristic, see for instance \cite{JMB}. In \cite{Pr}, the author refutes the Gel'fand-Kirillov Conjecture for the enveloping algebra of simple Lie algebras of certain types by actually refuting a modular version of the conjecture. This certainly shows that one should not restrict our attention only to the case where the characteristic is $0$, but also study the modular case. This motivated us to study the Poisson structure of fields of fractions of Poisson polynomial algebras over a field of arbitrary characteristic.

With the appearance of quantum groups in the eighties, new skew-fields of reference were needed, and a quantum version of the Gel'fand-Kirillov Conjecture was proposed by  Alev and Dumas \cite{AD}, and studied by numerous authors. We refer to \cite[I.2.11 and II.10.4]{BG} for information about this quantum version of the Gel'fand-Kirillov Conjecture.  In this context, skew-fields of reference are the skew-fields of fractions of quantum affine spaces.
	
Back to the Poisson setting, it is easy to build Poisson polynomial algebras whose fields of fractions are not Poisson isomorphic to Poisson-Weyl algebras. And so, as in the quantum case, we need to introduce other Poisson fields of reference as follows. A {\it Poisson affine field} is the field of fractions of a Poisson affine space, i.e. the field of fractions of a polynomial algebra in $n$ indeterminates $X_1,\dots,X_n$, with Poisson bracket given by $\{X_i,X_j\}=\la_{ij}X_iX_j$ for some skew-symmetric matrix $(\lambda_{ij}) \in M_n(\K)$. It was proved in \cite{GL} that Poisson-Weyl fields and Poisson affine fields are not isomorphic, so that Poisson affine fields were used in \cite{GL} as fields of reference for a Poisson version of the quantum Gel'fand-Kirillov Conjecture. Namely, the {\it quadratic Poisson Gel'fand-Kirillov problem} asks whether a given Poisson polynomial algebra is Poisson birationally equivalent to a Poisson affine space.  In \cite{GL}, it was shown that the fields of fractions of a large class of Poisson algebras are Poisson isomorphic to Poisson affine fields (over purely transcendental extensions of the base field). The method used to prove these Poisson isomorphisms is based on a Poisson version of the \textit{deleting derivation homomorphism} introduced by Cauchon in \cite{Cau} in order to prove the quantum Gel'fand-Kirillov Conjecture for a large class of noncommutative algebras, the so-called \textit{CGL extensions}. We note that, while Cauchon's deleting derivation homomorphism cannot be defined when the quantum parameter involved is a root of unity, Haynal \cite{HH} generalised Cauchon's construction to the root of unity case by using the notion of {\it higher derivation}.

The main aim of this paper is to establish the quadratic Poisson Gel'fand-Kirillov problem for a large class of Poisson polynomial algebras (and their quotients) over a field of arbitrary characteristic. Before explaining our strategy to attack this problem, we first give details on the Poisson algebras under consideration. 

Let $B$ be a Poisson algebra and $\alpha$ a Poisson derivation of $B$. 
Suppose that $\delta$ is a derivation on $B$ such that
$$\delta(\{a,b\})= \{\delta(a),b\}+ \{a,\delta(b)\}+ \alpha(a)\delta(b)-
\delta(a)\alpha(b)$$
for $a,b\in B$. By \cite[Theorem 1.1]{Oh} (after replacing our $B$
and $\alpha$ with $A$ and $-\alpha$), the Poisson structure on $B$ extends
uniquely to a Poisson algebra structure on the polynomial ring $A= B[X]$
such that
$$\{X,b\}= \alpha(b)X+ \delta(b)$$
 for $b\in B$. We write $A=
B[X;\alpha,\delta]_P$ to denote this situation, and we refer to $A$ as a
{\it Poisson-Ore extension} over $B$. In this article, we study iterated Poisson-Ore extensions. More precisely, we are concerned with polynomial algebras in several indeterminates $X_1,\dots,X_n$ over the base field $\K$, with Poisson bracket given by
	\[\{X_i,X_j\}=\la_{ij}X_iX_j+P_{ij} \ \ \  (j<i)\] 
where $(\lambda_{ij}) \in M_n(\K)$ is a skew-symmetric matrix and $P_{ij}$ is a polynomial in $X_1,\dots,X_{i-1}$ for all $j<i$. These Poisson algebras can be presented as iterated Poisson-Ore extensions over $\K$ of the form $\K[X_1][X_2;\alpha_2,\delta_2]_P \dots [X_n;\alpha_n,\delta_n]_P$. 
Examples of such Poisson algebras include for instance the semiclassical limits of quantum matrices, quantum  symplectic or euclidean spaces, quantum symmetric or antisymmetric matrices, etc.

In characteristic zero, the main tool used in \cite{GL} to establish the quadratic Poisson Gel'fand-Kirillov problem for a large class of Poisson polynomial algebras (and their quotients) is the so-called Poisson deleting derivation homomorphism. This homomorphism is a Poisson algebra isomorphism between localisations of two Poisson-Ore extensions:
\begin{align*}
	F:A[Y^{\pm1};\al]_P&\stackrel{\cong}\longrightarrow A[X^{\pm1};\al,\de]_P\\
	A\ni a \qquad &\longmapsto \sum_{i\geq0}\Big(\frac{-1}{s}\Big)^i\frac{\de^i(a)}{i!}X^{-i},\\
	Y\qquad &\longmapsto X.
\end{align*}
under the assumptions that the derivation $\delta$ is locally nilpotent and $\alpha \delta = \delta (\alpha + s)$ for some $s \in \K^{\times}$. 

Obviously the above formula defining the Poisson deleting derivation homomorphism does not make sense in positive characteristic due to the division by $i!$. To overcome this problem and define a characteristic-free Poisson deleting derivation homomorphism, we observe that the sequence of linear maps $\big(\frac{\de^i}{i!}\big)$ is a so-called iterative higher derivation (see Definition \ref{hd}) which extends the derivation $\delta$ (that is, whose first terms are $\mathrm{id}$ and $\delta$). In Section 2, we construct a characteristic-free Poisson deleting derivation homomorphism in the case where the derivation $\delta$ extends to a so-called iterative {\it higher Poisson derivation}, i.e. an iterative higher derivation compatible with the Poisson structure (see Section 2.2). We also study the compatibility of this characteristic-free Poisson deleting derivation homomorphism with torus actions (Section 2.4). In Section 3, we use the characteristic-free Poisson deleting derivation homomorphism repeatedly to prove that the quadratic Poisson Gel'fand-Kirillov problem holds for a large class of iterated Poisson-Ore extensions. We actually prove a stronger result  by considering also Poisson prime quotients. More precisely, we show that if $P$ is a Poisson prime ideal of a Poisson polynomial algebra $A$ to which our construction applies,  then there exists a Poisson prime ideal $Q$ in a Poisson affine space $B$ such that $\F (A/P)\cong \F (B/Q)$ as Poisson algebras (this proves the quadratic Poisson Gel'fand-Kirillov problem for $A$ since $Q=0$ when $P=0$). Additionally, if a torus $H$ is acting rationally by Poisson automorphisms on $A$ and if $P$ is invariant under this action, we show, modulo some technical assumptions, that the ideal $Q$ of the Poisson affine space $B$ is also invariant under the induced torus action on $B$. Under certain mild assumptions on the base field, we prove that $B$ has only finitely many $H$-invariant Poisson prime ideals and that they are all generated by some of its  generators. As a consequence, when $P$ is $H$-invariant, the quotient $B/Q$ is a Poisson affine space, so that the quotient $A/P$ also satisfies the quadratic Poisson Gel'fand-Kirillov problem. In other words, our main theorem reads as follows.

\begin{thm}[Theorem \ref{all hypo}]
	Let $A=K[X_1][X_2;\alpha_2,\delta_2]_P \dots [X_n;\alpha_n,\delta_n]_P$ be an iterated Poisson-Ore extension over a field $\K$ of arbitrary characteristic. Assume that the torus $H=(\K^{\times})^r$ acts rationally by Poisson automorphisms on $A$ (and that the hypotheses of Theorem \ref{abc} and Hypothesis \ref{hyp} are satisfied). Then, for any $H$-invariant Poisson prime ideal $P$ of $A$, the field $\F A/P$ is Poisson isomorphic to a Poisson affine field.
\end{thm}

Contrary to the characteristic zero case, there is one hypothesis in Theorem 1.1 that is difficult to check. Namely, the existence of iterative higher Poisson derivations extending the derivations $\delta_i$. In characteristic zero, the only iterative higher Poisson derivation extending a derivation $\delta$ is actually the {\it canonical higher derivation} $(\frac{\de^i}{i!})$. In prime characteristic, the existence of an iterative higher Poisson derivation extending a given derivation is a harder problem. In Section 4 we tackle this problem using the so-called \textit{semiclassical limit} process (see Section 4.1 for details). More precisely, we show that the existence of a quantum version of the canonical higher derivation in a ``quantum algebra'' $A$ ensures (under mild hypothesis) the existence of a higher Poisson derivation in the semiclassical limit of $A$ (see Theorem \ref{fin} in Section 4.1). At the noncommutative level, the characteristic of the base field does not influence the existence of quantum version of the canonical higher derivation. The existence only depends on the genericity of the deformation parameter. However, in our case, the deformation parameter is always transcendental (to allow for the semiclassical limit process), thus ensuring the existence of quantum canonical higher derivations. 
As a consequence, we obtain many examples of Poisson algebras to which Theorem 1.1 applies. For instance, we obtain that the coordinate rings of Poisson matrix varieties and their $H$-invariant Poisson prime quotients, such as  the coordinate rings of determinantal varieties, satisfy the quadratic Poisson Gel'fand-Kirillov problem over a field of characteristic different of $2$ (see Sections 4.2 and 4.3).

\section{Poisson deleting derivation homomorphism}

The main aim of this section is to extend the Poisson deleting derivation homomorphism defined in characteristic 0 in \cite{GL} to the prime characteristic case. We first define the class of Poisson algebras that we are concerned with in this article, the so-called \textit{Poisson polynomial algebras} or \textit{iterated Poisson-Ore extensions} (Section 2.1). Then we introduce the notion of higher Poisson derivation in Section 2.2. As explained in the introduction we use these higher Poisson  derivations to overcome the characteristic problem, and thus define the characteristic-free Poisson deleting derivation homomorphism in Section 2.3. Section 2.4 is concerned with the compatibility of the characteristic-free Poisson deleting derivation homomorphism and the action of a torus acting rationally by Poisson automorphisms on the Poisson-Ore extension under consideration. This will be used later to prove the quadratic Poisson Gel'fand-Kirillov problem for torus-invariant prime factors of certain iterated Poisson-Ore extensions (see Section 3.3).

\subsection{Poisson-Ore extension}

	Poisson-Ore extensions are Poisson analogue of the well-known notion of Ore extension, or skew polynomial ring, in noncommutative ring theory. Their definition is based on the following result of Oh \cite[Theorem 1.1]{Oh}.
	
\begin{thm}[Oh]
\label{Oh}
Let $\al$ and $\de$ be $\K$-linear maps of a Poisson $\K$-algebra $A$. Then the polynomial algebra $R=A[X]$ is a Poisson algebra with Poisson bracket extending the Poisson bracket of $A$ and satisfying 
\begin{center}
$\{X,a\}=\al(a)X+\de(a)$ for all $a\in A$,
\end{center}
 if and only if $\al$ is a Poisson derivation of $A$, i.e. $\al$ is a $\K$-derivation of $A$ with 
\begin{center}
$\al(\{a,b\})=\{\al(a),b\}+\{a,\al(b)\}$ for all $a,b\in A$,
\end{center}
and 
$\de$ is a Poisson $\al$-derivation of $A$, i.e. $\de$ is a $\K$-derivation of $A$ with 
\begin{center}
$\de(\{a,b\})=\{\de(a),b\}+\{a,\de(b)\}+\al(a)\de(b)-\de(a)\al(b)$ for all $a,b\in A$.
\end{center}
\end{thm}

\begin{Def}
Let $A$ be a Poisson algebra. The set of Poisson derivations of $A$ is denoted by $\D_P(A)$. Let $\al\in\D_P(A)$ and $\de$ be a Poisson $\al$-derivation of $A$. Set $R=A[X]$. The algebra $R$ endowed with the Poisson bracket from Theorem \ref{Oh} is denoted by $R=A[X;\al,\de]_P$ and called \textit{Poisson-Ore extension}. As usual we write $A[X;\al]_P$ for $A[X;\al,0]_P$.
\end{Def}

This construction is easily iterated. We say that \textit{$R$ is an iterated Poisson-Ore extension over $A$} if 
	\[R=A[X_1;\al_1,\de_1]_P[X_2;\al_2,\de_2]_P\cdots[X_n;\al_n,\de_n]_P\]
for some Poisson derivations $\al_1,\dots,\al_n$ and $\al_i$-Poisson derivations $\de_i$ ($1\leq i\leq n$) of the appropriate Poisson subalgebras.

	Let $\bo\la=(\lambda_{ij}) \in M_n(\K)$ be a skew-symmetric matrix. Then the polynomial algebra $\K[X_1,\dots,X_n]$ is a Poisson algebra with Poisson bracket defined by $\{X_i,X_j\}=\la_{ij}X_iX_j$ for all $i,j$. This Poisson algebra is called a \textit{Poisson affine $n$-space} and is denoted by $\K_{\bo\la}[X_1,\dots,X_n]$. This quadratic Poisson structure extends (uniquely) to a Poisson bracket on $\K(X_1,\dots,X_n)$. The field $\K(X_1,\dots,X_n)$ endowed with this Poisson structure is called a \textit{Poisson affine field} and is denoted by $\K_{\bo\la}(X_1,\dots,X_n)$. Finally, this quadratic Poisson bracket on the Poisson affine space $\K[X_1,\dots,X_n]$ extends uniquely to a Poisson bracket on the Laurent polynomial algebra $\K[X_1^{\pm1},\dots,X_n^{\pm1}]$. We call this Poisson algebra a \textit{Poisson torus} and it is denoted by $\K_{\bo\la}[X_1^{\pm1},\dots,X_n^{\pm1}]$.
	
	It is clear that the Poisson affine $n$-space $\K_{\bo\la}[X_1,\dots,X_n]$ is an iterated Poisson-Ore extension of the form $\K[X_1][X_2;\al_2]_P\cdots[X_n;\al_n]_P$, where $\al_i$ is the Poisson derivation of $\K[X_1,\dots,X_{i-1}]$ such that $\al_i(X_j)=\la_{ij}X_j$ for all $1\leq j<i\leq n$.

\subsection{Higher Poisson derivations}

	The main tool to build Poisson birational isomorphism between (certain) iterated Poisson-Ore extensions and Poisson affine $n$-spaces is the existence of higher derivations which are compatible with Poisson brackets. We now fix the notation and terminology used in this article.
\begin{Def}
\label{hd}
Let $A$ be a Poisson $\K$-algebra, $\al\in\D_P(A)$ and $\eta\in\K$.
\begin{enumerate}
	\item A \textit{higher derivation} on $A$ is a sequence of $\K$-linear maps $(D_i)_{i=0}^{\infty}=(D_i)$ such that:
		\begin{center}
\ \ \ $D_0=\id_A$ and $D_n(ab)=\sum\limits_{i=0}^{n}D_i(a)D_{n-i}(b)$ for all $a,b\in A$ and all $n\geq0$. \qquad \ \ (A1)
		\end{center}
A higher derivation is \textit{iterative} if $D_iD_j=\binom{i+j}{i}D_{i+j}$ for all $i,j\geq0$, and \textit{locally nilpotent} if for all $a\in A$ there exists $n\geq0$ such that $D_i(a)=0$ for all $i\geq n$.
	\item A higher derivation $(D_i)$ is a \textit{higher $\al$-skew Poisson derivation} if for all $a,b\in A$	and all $n\geq0$:
		\begin{center}
\ \ \ $D_n(\{a,b\})=\sum\limits_{i=0}^{n}\{D_i(a),D_{n-i}(b)\}+i\big(\al D_{n-i}(a)D_i(b)-D_i(a)\al D_{n-i}(b)\big)$. \qquad \ (A2)
		\end{center}
	\item A higher $\al$-skew Poisson derivation is a \textit{higher $(\eta,\al)$-skew Poisson derivation} if for all $i\geq0$:
		\begin{center}
\qquad \qquad \qquad \qquad \qquad \qquad \qquad $D_i\al=\al D_i+i\eta D_i$.   \qquad \qquad \qquad \qquad \qquad \ \ \ \ \ (A3)
		\end{center}
	\item We say that the derivation $\de$ of a Poisson-Ore extension $A[X;\al,\de]_P$ \textit{extends to a higher $(\eta,\al)$-skew Poisson derivation} if there exists a higher $(\eta,\al)$-skew Poisson derivation $(D_i)$ on $A$ such that $D_1=\de$.
	\end{enumerate}
\end{Def}
\begin{rems}
Let $A$ be a $\K$-algebra and $\de$ a derivation on $A$.
\begin{enumerate}
	\item In characteristic $0$, the only iterative higher derivation $(D_i)$ on $A$ such that $D_1=\de$ is given by:
	\[D_n=\frac{\de^n}{n!}\]
	for all $n\geq0$ (this easily follows from \cite[Proposition 2.1]{Mad}). This iterative higher derivation is called the \textit{canonical higher derivation associated to $\de$}.
	\item In characteristic $p>0$, an iterative higher derivation $(D_i)$ is uniquely determined by the $D_{p^k}$ for $k\geq0$. More precisely for $n=\sum_{k=0}^m n_kp^k$, the $p$-adic decomposition of $n$ we have:
	\[D_n=\frac{D_1^{n_0}D_p^{n_1}\cdots D_{p^m}^{n_m}}{n_0!n_1!\cdots n_m!}.\]
See \cite{Weis}, the result for fields being trivially adapted for $\K$-algebras.
\end{enumerate}
\end{rems}

\begin{ex}
\label{exgl}
	Suppose $\K$ of characteristic zero. Let $R=A[X;\al,\de]_P$ be a Poisson-Ore extension where $A$ is a Poisson $\K$-algebra. If there exists $\eta\in\K^{\times}$ such that $\de\al=\al\de+\eta\de$ then it follows from \cite[Lemma 3.6]{GL} (with $s=-\eta$) that:
	\[\de^n(\{a,b\})=\sum_{l+m=n}\binom{n}{l}\big(\{\de^l(a),\de^m(b)\}+m\de^l(a)\al\de^m(b)-l\de^l(a)\de^m\al(b)\big)\]
for all $a,b\in A$ and all $n\geq0$. From this it is easily shown that the canonical higher derivation $\Big(\frac{\de^n}{n!}\Big)$ is an iterative higher $(\eta,\al)$-skew Poisson derivation. The examples given in \cite{GL} provide a large family of $\al$-derivations $\de$ satisfying $\de\al=\al\de+\eta\de$ for some scalar $\eta\in \K^{\times}$, which extend to higher $(\eta,\al)$-skew Poisson derivations.
\end{ex}

	The following proposition gives a criterion for a sequence of $\K$-linear maps to be a higher $(\eta,\al)$-skew Poisson derivation. This will be used later to extend a higher $(\eta,\al)$-skew Poisson derivation to certain localisations. For $\be\in\D_P(A)$, the Poisson bracket of $A$ uniquely extends to a Poisson bracket on the formal power series algebra $A[[X]]$ by setting $\{X,a\}=\be(a)X$. This Poisson algebra is denoted by $A[[X;\be]]_P$. The Poisson bracket of two elements of $A[[X;\be]]_P$ is given by:
	\[\{\sum_{i\geq 0}a_iX^i,\sum_{j\geq 0}b_jX^j\}=\sum_{n\geq0}\Big(\sum_{i+j=n}\big(\{a_i,b_j\}+ia_i\be(b_j)-j\be(a_i)b_j\big)\Big)X^n,\]
where all the $a_i$s and the $b_j$s are in $A$. (Remark that we have just extended by continuity the Poisson bracket of $A[X;\be]_P$ to its completion $A[[X]]$.) Note that the Poisson derivation $\be$ of $A$ extends to a Poisson derivation of $A[[X;\be]]_P$ by setting $\be(X)=\eta X$ for any $\eta\in\K$ since:
	\[\be(\{X,a\})=\big(\be^2(a)+\eta\be(a)\big)X=\{\be(X),a\}+\{X,\be(a)\}.\]

\begin{prop}
\label{diag}
Let $(D_i)_{i=0}^\infty$ be a sequence of $\K$-linear maps on a Poisson $\K$-algebra $A$ with $D_0=\id_A$, $\al\in\D_P(A)$ and $\eta\in\K$.
\begin{enumerate}
	\item [\rm{(a)}] $(D_i)$ is a higher $\al$-skew Poisson derivation on $A$ if and only if the $\K$-linear map $\Psi : A\rightarrow A[[X;-\al]]_P$ given by $a\mapsto\sum_{i=0}^{\infty}D_i(a)X^i$ is a Poisson homomorphism.
	\item [\rm{(b)}] Extend $\al$ to a Poisson derivation on $A[[X;-\al]]_P$ by setting $\al(X)=\eta X$. Assume that $(D_i)$ is a higher $\al$-skew Poisson derivation. Then $(D_i)$ is a higher $(\eta,\al)$-skew Poisson derivation if and only if the diagram of Figure 1 is commutative.
 \end{enumerate}

	\begin{figure}[!ht]
	\centering
		\begin{tikzpicture}
	\node (a) at (0,0){$A[[X;-\al]]_P$};
	\node (b) at (3,0){$A[[X;-\al]]_P$};
	\node (c) at (0,-2){$A$};
	\node (d) at (3,-2){$A$};
	\node (aa) at (1.5,0.2){$\al$};
	\node (bb) at (-0.2,-1){$\Psi$};
	\node (cc) at (1.5,-2.2){$\al$};
	\node (dd) at (3.2,-1){$\Psi$};
	\draw [->] (a) to (b);\draw [->] (c) to (d);
	\draw [->] (c) to (a);\draw [->] (d) to (b);
		\end{tikzpicture}
		\caption{}
	\label{diag1}
 	\end{figure}
\end{prop}

\begin{proof}
(a) It is obvious that $\Psi$ is a $\K$-algebra homomorphism if and only if $(D_i)$ satisfies Axiom (A1). Let $a,b\in A$. We need to check that the equality $\Psi(\{a,b\})=\{\Psi(a),\Psi(b)\}$ is equivalent to Axiom (A2):
\begin{align*}
	\{\Psi(a),\Psi(b)\}&=\sum_{i,j}\{D_i(a)X^i,D_j(b)X^j\}\\
	&=\sum_{i,j}\{D_i(a),D_j(b)\}X^{i+j}-iD_i(a)\al D_j(b)X^{i+j}+j\al D_i(a)D_j(b)X^{i+j}\\
	&=\sum_{i,j}\Big(\{D_i(a),D_j(b)\}+j\al D_i(a)D_j(b)-iD_i(a)\al D_j(b)\Big)X^{i+j}\\
	&=\sum_{n\geq0}\sum_{i+j=n}\Big(\{D_i(a),D_j(b)\}+i\al D_j(a)D_i(b)-iD_i(a)\al D_j(b)\Big)X^n.
	\end{align*}
 Since $\Psi(\{a,b\})=\sum_{n\geq0}D_n(\{a,b\})X^n$ and $\{X^n\ |\ n\geq0\}$ is a basis of $A[[X]]$, the equivalence is shown.

(b) We show that $\Psi\al=\al\Psi$ is equivalent to  Axiom (A3). Let $a\in A$. Then we have:
\begin{align*}
	\al\Psi(a)&=\sum_{i\geq0}\al\big(D_i(a)X^i\big)\\
	&=\sum_{i\geq0}\al D_i(a)X^i+D_i(a)\al(X^i)\\
	&=\sum_{i\geq0}\big(\al D_i(a)+i\eta D_i(a)\big)X^i,
\end{align*}
On the other hand, we have:
\begin{align*}
	\Psi\al(a)=\sum_{i\geq0}\big(D_i\al(a)\big)X^i.
\end{align*}
Hence $\Psi\al=\al\Psi$ if and only if $(D_i)$ satisfies Axiom (A3).
\end{proof}

\begin{prop}
Let $\al\in\D_P(A)$, $\eta\in\K$ and $(D_i)$ a higher $(\eta,\al)$-skew Poisson derivation on a Poisson $\K$-algebra $A$. Let $S$ be a multiplicative set of regular elements of $A$. Then $(D_i)$ uniquely extends to a higher $(\eta,\al)$-skew Poisson derivation on $AS^{-1}$.
\end{prop}

\begin{proof}
A derivation $\be$ of $A$ extends uniquely to $AS^{-1}$ by:
\begin{align}
	\be(as^{-1})=\be(a)s^{-1}-as^{-2}\be(s) \text{ for $a\in A$ and $s\in S$}.
\end{align}
So we can extend uniquely $\al$ and $D_1$ to $AS^{-1}$. Moreover if $\al\in\D_P(A)$ then after extension $\al\in\D_P(AS^{-1})$.\\
Now suppose that $(D_i)$ extends to a higher ($\eta,\al$)-skew Poisson derivation on $AS^{-1}$. For $a\in A$ and $s\in S$, we apply $D_n$ to the equation $a1^{-1}=(as^{-1})(s1^{-1})$ to get:
\begin{align*}
	D_n(a)1^{-1}&=D_n\big((as^{-1})(s1^{-1})\big)\\
	&=\sum\limits_{i=0}^{n}D_i(as^{-1})D_{n-i}(s1^{-1})\\
	&=D_n(as^{-1})s1^{-1}+\sum\limits_{i=0}^{n-1}D_i(as^{-1})D_{n-i}(s1^{-1}).
\end{align*}
This implies:
	\[D_n(as^{-1})=\Big(D_n(a)-\sum\limits_{i=0}^{n-1}D_i(as^{-1})D_{n-i}(s)\Big)s^{-1}\]
and proves the unicity.

	Let $\Psi : A\rightarrow A[[X;-\al]]_P$ be the $\K$-linear map defined in Proposition \ref{diag} and let \\
$\Phi :A[[X;-\al]]_P\rightarrow AS^{-1}[[X;-\al]]_P$ be the canonical embedding. Consider the composite map $\Gamma=\Phi\circ\Psi : A\rightarrow AS^{-1}[[X;-\al]]_P$ and note that $\Gamma$ is a $\K$-algebra Poisson homomorphism by Proposition \ref{diag}, since $(D_i)$ is a higher $\al$-skew Poisson derivation on $A$. For all $s\in S$, the constant term of $\Gamma(s)$ is a unit in $AS^{-1}$ and so $\Gamma(s)$ is a unit in $AS^{-1}[[X;-\al]]_P$. Hence $\Gamma$ extends to a $\K$-algebra homomorphism $\Gamma ': AS^{-1}\rightarrow AS^{-1}[[X;-\al]]_P$ such that $\Gamma'(as^{-1})=\Gamma(a)\Gamma(s)^{-1}$. A straigthforward computation shows that $\Gamma'$ is a Poisson homomorphism.
	
	We consider the diagram of Figure \ref{diag2},
\begin{figure}[!ht]
	\centering
		\begin{tikzpicture}
	\node (a) at (0,0){$AS^{-1}[[X;-\al]]_P$};
	\node (b) at (4,0){$AS^{-1}[[X;-\al]]_P$};
	\node (c) at (0,-2){$AS^{-1}$};
	\node (d) at (4,-2){$AS^{-1}$};
	\node (aa) at (2,0.2){$\al$};
	\node (bb) at (-0.4,-1){$\Gamma '$};
	\node (cc) at (2,-2.2){$\al$};
	\node (dd) at (4.3,-1){$\Gamma '$};
	\draw [->] (a) to (b);\draw [->] (c) to (d);
	\draw [->] (c) to (a);\draw [->] (d) to (b);
		\end{tikzpicture}
		\caption{}
	\label{diag2}
\end{figure}
where $\al$ has been extended to a Poisson derivation of $AS^{-1}[[X;-\al]]_P$ via (1) and $\al(X)=\eta X$. Since $\Gamma(a)=\sum_{i\geq0}(D_i(a)1^{-1})X^i$, and $(D_i)$ is a higher $(\eta,\al)$-skew Poisson derivation on $A$ we have:
\begin{align*}
	\al\Gamma(a)&=\sum_{i\geq0}\al\big((D_i(a)1^{-1})X^i\big)\\
	&=\sum_{i\geq0}\al(D_i(a)1^{-1})X^i+(D_i(a)1^{-1})\al(X^i)\\
	&=\sum_{i\geq0}(\al D_i(a)1^{-1}+i\eta D_i(a)1^{-1})X^i\\
	&=\sum_{i\geq0}(D_i\al(a)1^{-1})X^i=\Gamma\al(a)\quad \qquad \text{for all $a\in A$.}
\end{align*}
Since $\Gamma$ is a $\K$-algebra homomorphism and $\al$ a $\K$-derivation we have:
\begin{align*}
	\al\Gamma '(as^{-1})&=\al(\Gamma(a)\Gamma(s)^{-1})\\
	&=\al\Gamma(a)\Gamma(s)^{-1}-\Gamma(a)\Gamma(s)^{-2}\al\Gamma(s)\\
	&=\Gamma\al(a)\Gamma(s)^{-1}-\Gamma(a)\Gamma(\al(s))\Gamma(s)^{-2}\\
	&=\Gamma\al(a)\Gamma(s)^{-1}-\Gamma(a\al(s))\Gamma(s^2)^{-1}\\
	&=\Gamma '(\al(a)s^{-1}-a\al(s)s^{-2})\\
	&=\Gamma'\al(as^{-1}).
\end{align*}
Thus the diagram of Figure \ref{diag2} is commutative, as desired.

	Define a sequence $(D_i)$ on $AS^{-1}$ such that $D_i(as^{-1})$ is the coefficient of $X^i$ in $\Gamma '(as^{-1})$ for all $as^{-1}\in AS^{-1}$. Then, by Proposition \ref{diag}, we conclude that this sequence is a higher ($\eta,\al$)-skew Poisson derivation on $AS^{-1}$ extending $(D_i)$ on $A$, as requested.
\end{proof}

	We conclude this section by two easy technical lemmas whose proofs are left to the reader.

\begin{lem}
\label{sub}
Let $A$ be a Poisson $\K$-algebra, let $B\subseteq A$ be a Poisson subalgebra generated, as an algebra, by a finite set $\{b_1,\dots,b_k\}$. Let $\al\in\D_P(A)$ and $(D_i)$ be a higher $\al$-skew Poisson derivation on $A$. If $D_i(b_j)\in B$ and $\al(b_j)\in B$ for all $i\geq0$ and all $1\leq j\leq k$, then $D_n(B)\subseteq B$ and $D_n(\{B,B\})\subseteq B$ for all $n\geq0$.
\end{lem}

\begin{lem}
\label{ln}
Let $A$ be a commutative $\K$-algebra generated by a finite set $\{a_1,\dots,a_k\}$. Let $\al\in\D(A)$ and $(D_i)$ be a higher derivation on $A$. If $(D_i)$ is locally nilpotent on $a_j$ for all $1\leq j\leq k$, then $(D_i)$ is locally nilpotent on $A$.
\end{lem}

\subsection{Deleting derivation homomorphism}
\label{ddh}

Let $A[X;\al,\de]_P$ be a Poisson-Ore extension, where $A$ is a Poisson $\K$-algebra and set $S=\{X^n\ |\ n\geq0\}$. The set $S$ is a multiplicative set (of regular elements) and we denote by $A[X^{\pm1};\al,\de]_P$ the localisation $S^{-1}\big(A[X;\al,\de]_P\big)$. Poisson brackets extend uniquely by localisation, so $A[X^{\pm1};\al,\de]_P$ is also a Poisson algebra, called \textit{Poisson-Ore Laurent algebra}. Suppose that the derivation $\de$ extends to an iterative locally nilpotent higher $(\al,\eta)$-skew Poisson derivation $(D_i)$ with $\eta\in\K^{\times}$. We define the map $F: A\rightarrow A[X^{\pm1};\al,\de]_P$ by
	\[F(a)=\sum_{i\geq0}\frac{1}{\eta^i}D_i(a)X^{-i} \qquad \text{for all $a\in A$.}\]
Note that this sum is finite since $(D_i)$ is locally nilpotent.
\begin{prop}
\label{ddm}
 The $\K$-linear map $F :A\rightarrow A[X^{\pm1};\al,\de]_P$ is a Poisson homomorphism and satisfies the following identity
	\[\{X,F(a)\}=F\big(\al(a)\big)X\qquad\text{for all $a\in A$}.\]
\end{prop}
\begin{proof}
$F$ is an algebra homomorphism because $(D_i)$ satisfies Axiom (A1).
Let us show that $F$ respects the Poisson bracket using Axiom (A2) and the iterativity of $(D_i)$.

\begin{align*}
	\{F(a),F(b)\}&=\sum_{i,j\geq0}\{\frac{1}{\eta^i}D_{i}(a)X^{-i},\frac{1}{\eta^j}D_{j}(b)X^{-j}\}\\
	&=\sum_{i,j\geq0}\frac{1}{\eta^{i+j}}\Big(\{D_i(a),D_j(b)\}X^{-i-j}+D_j(b)\{D_i(a),X^{-j}\}X^{-i}+D_{i}(a)\{X^{-i},D_j(b)\}X^{-j}\Big)\\
	&=\sum_{i,j\geq0}\frac{1}{\eta^{i+j}}\Big(\{D_i(a),D_j(b)\}X^{-i-j}+jD_j(b)\big(\al D_i(a)X+D_1D_i(a)\big)X^{-i-j-1}\\
	&\qquad\qquad\qquad\qquad\qquad\qquad\qquad-iD_i(a)\big(\al D_j(b)X+D_1D_j(b)\big)X^{-i-j-1}\Big)\\
	&=\sum_{i,j\geq0}\frac{1}{\eta^{i+j}}\Big(\{D_i(a),D_j(b)\}+j\al D_i(a)D_j(b)-iD_i(a)\al D_j(b)\Big)X^{-i-j}\\
	&\quad+\sum_{i,j\geq0}\frac{1}{\eta^{i+j}}\Big(jD_j(b)D_1D_i(a)-iD_i(a)D_1D_j(b)\Big)X^{-i-j-1}\\
	&=\sum_{i,j\geq0}\frac{1}{\eta^{i+j}}\Big(\{D_i(a),D_j(b)\}+i\al D_j(a)D_i(b)-iD_i(a)\al D_j(b)\Big)X^{-i-j}\\
	&\quad+\sum_{i,j\geq0}\frac{1}{\eta^{i+j}}\Big(j(i+1)D_j(b)D_{i+1}(a)-i(j+1)D_i(a)D_{j+1}(b)\Big)X^{-i-j-1}\\
	&=\sum_{t\geq0}\frac{1}{\eta^t}\sum_{i+j=t}\Big(\{D_i(a),D_j(b)\}+i\big(\al D_j(a)D_i(b)-D_i(a)\al D_j(b)\big)\Big)X^{-t}\\
	&\quad+\sum_{j,l\geq1}\frac{jl}{\eta^{j+l-1}}D_j(b)D_l(a)X^{-j-l}-\sum_{i,k\geq1}\frac{ik}{\eta^{i+k-1}}D_i(a)D_k(b)X^{-i-k}\\
	&=\sum_{t\geq0}\frac{1}{\eta^t}D_t(\{a,b\})X^{-t}\\
	&=F(\{a,b\}).
\end{align*}

	Finally we use Axiom (A3) and the iterativity of $(D_i)$ to show that $\{X,F(a)\}=F\big(\al(a)\big)X$. Indeed, we have:
\begin{align*}
	\{X,F(a)\}&=\sum_{i\geq0}\frac{1}{\eta^i}\{X,D_i(a)\}X^{-i}\\
	&=\sum_{i\geq0}\frac{1}{\eta^i}\big(\al \big(D_i(a)\big)X+D_1D_i(a)\big)X^{-i}\\
	&=\sum_{i\geq0}\frac{1}{\eta^i}\al \big(D_i(a)\big)X^{-i+1}+\sum_{i\geq0}\frac{1}{\eta^i}(i+1)D_{i+1}(a)X^{-i}\\
	&=\sum_{i\geq0}\frac{1}{\eta^i}\al \big(D_i(a)\big)X^{-i+1}+\sum_{i\geq1}\frac{\eta}{\eta^i}iD_i(a)X^{-i+1}\\
	&=\sum_{i\geq0}\frac{1}{\eta^i}\big(\al \big(D_i(a)\big)+i\eta D_i(a)\big)X^{-i+1}\\
	&=\sum_{i\geq0}\frac{1}{\eta^i}D_i\big(\al(a)\big)X^{-i+1}\\
	&=F\big(\al(a)\big)X.
\end{align*}
\end{proof}

We are now ready to state the main result of this section.

\begin{thm}
\label{iso}
Let $A[X;\al,\de]_P$ be a Poisson-Ore extension, where $A$ is a Poisson $\K$-algebra. Suppose that $\de$ extends to an iterative, locally nilpotent higher $(\eta,\al)$-skew Poisson derivation $(D_i)$ on $A$ with $\eta\in\K^{\times}$. Then the algebra homomorphism $F: A\rightarrow A[X^{\pm1}]$ defined by:
	\[F(a)=\sum_{i\geq 0}\frac{1}{\eta^i}D_i(a)X^{-i}\]
uniquely extends to a Poisson $\K$-algebra isomorphism:
	\[F: A[Y^{\pm1};\al]_P\stackrel{\cong}{\longrightarrow} A[X^{\pm1};\al,\de]_P\]
by setting $F(Y)=X$.
\end{thm}
\begin{proof}
	Clearly $F$ extends uniquely to a $\K$-algebra homomorphism from $A[Y^{\pm1}]$ to $A[X^{\pm1}]$ by setting $F(Y)=X$. In view of Proposition \ref{ddm} we know  that $F(\{a,b\})=\{F(a),F(b)\}$ for all $a,b\in A$. Moreover, we have:
	\[F(\{Y,a\})=F(\al(a)Y)=F\big(\al(a)\big)F(Y)=F\big(\al(a)\big)X=\{X,F(a)\}=\{F(Y),F(a)\}\]
for all $a\in A$. Thus $F$ is a Poisson homomorphism from $A[Y^{\pm1};\al]_P$ to $A[X^{\pm1};\al,\de]_P$. 

	To conclude we show that $F$ is bijective. First, let $f\in A[Y^{\pm1}]$ be a nonzero Laurent polynomial. We can write $f=\sum_{i=l}^{m}a_iY^i$, where $l,m$ are two integers with $l\leq m$, and $a_i\in A$ for all $i\in\{l,\dots,m\}$ with $a_m\neq0$. Observing that 
	\[F(a_iY^i)=\sum_{k\geq0}\frac{1}{\eta^k}D_k(a_i)X^{i-k}=a_iX^i+\sum_{k\geq1}\frac{1}{\eta^k}D_k(a_i)X^{i-k}\]
for all $i$, we can write $F(f)=a_mX^m+\sum_{i=j}^{m-1}b_iX^i$, for some $j<m$ and where $b_i\in A$ for all $j\leq i<m$. Thus $F(f)\neq0$, and $F$ is injective.
	
 	 For the surjectivity, we already have $F(Y^{\pm1})=X^{\pm1}$, so we just need to check that $A\subset \text{Im}(F)$. Let $a\in A$. Since $(D_i)$ is locally nilpotent, there exists $l\geq 0$ such that $D_l(a)=0$. If $l\leq 1$, we have $F(a)=a$ and so $a\in \text{Im}(F)$. Assume $l>1$ and write $F(a)=a+\sum_{i=1}^{l-1}\frac{1}{\eta^i}D_i(a)X^{-i}$. Since $D_{l-i}D_{i}(a)=\binom{l}{i}D_{l}(a)=0$ for $i=1,\dots,l-1$, we have $D_{i}(a)\in\text{Im}(F)$ for all $i=1,\dots,l-1$ (we proceed by induction on $l$). Thus $F(a)-a$ is in the image of $F$ and so does $a$. Thus $F$ is surjective.
\end{proof}

\subsection{Case where a torus acts rationally: $H$-equivariance of the deleting derivation homomorphism}
\label{H action}

	Let $A$ be a finitely generated Poisson $\K$-algebra. Suppose that a torus $H$ is acting by Poisson $\K$-algebra automorphisms on a Poisson-Ore extension $A[X;\al,\de]_P$ such that $H(A)=A$. We suppose that the indeterminate $X$ is an $H$-eigenvector and that $H$ commutes with the derivation $\al$. Let $h\in H$ and set $h(X)=\mu X$ for a scalar $\mu\in\K^{\times}$. Then $H$ is also acting by automorphisms on $A[Y;\al]_P$ via:
	\[h\big(\sum_{i=0}^{n}a_iY^i\big)=\sum_{i=0}^{n}h(a_i)\mu^{i}Y^i\]
for all $h\in H$. Note that $h(Y)=\mu Y$. Moreover this action respects the Poisson bracket of $A[Y;\al]_P$ since
	\[h(\{Y,a\})=h(\al(a)Y)=h(\al(a))h(Y)=\mu\al(h(a))Y=\mu\{Y,h(a)\}=\{h(Y),h(a)\}.\]
		
	These $H$-actions extend uniquely by localisation on $A[X^{\pm1};\al,\de]_P$ and $A[Y^{\pm1};\al]_P$ since $X$ and $Y$ are $H$-eigenvectors. With a desire of clarity, we sometimes distinguish between the actions of $h\in H$ on $A[X^{\pm1};\al,\de]_P$ and $A[Y^{\pm1};\al]_P$ by using subscripts: $h_X$ and $h_Y$. The following lemma gives conditions under which these actions commute with the deleting derivation homomorphism $F$ defined at the beginning of Section \ref{ddh}.
\begin{figure}[!ht]
	\centering
		\begin{tikzpicture}
	\node (a) at (0,0){$A[X^{\pm1};\al,\de]_P$};
	\node (b) at (4,0){$A[X^{\pm1};\al,\de]_P$};
	\node (c) at (0,-2){$A[Y^{\pm1};\al]_P$};
	\node (d) at (4,-2){$A[Y^{\pm1};\al]_P$};
	\node (aa) at (2,0.2){$h_X$};
	\node (bb) at (-0.4,-1){$F$};
	\node (cc) at (2,-2.2){$h_Y$};
	\node (dd) at (4.3,-1){$F$};
	\draw [->] (a) to (b);\draw [->] (c) to (d);
	\draw [->] (c) to (a);\draw [->] (d) to (b);
		\end{tikzpicture}
		\caption{}
	\label{diag3}
\end{figure}

\begin{lem}
\label{com}
Suppose that $\de$ extends to a higher $(\eta,\al)$-skew Poisson derivation $(D_i)$ on $A$ with $\eta\in\K^{\times}$. We denote by $\{a_1,\dots,a_l\}$ a set of generators of $A$. If for all $n\geq0$ and all $1\leq i\leq l$ we have 
	\[h\big(D_n(a_i)\big)=\mu^nD_n\big(h(a_i)\big)\]
then $h_XF=Fh_Y$, that is the diagram of Figure \ref{diag3} is commutative.
\end{lem}
\begin{proof}
For all $1\leq i\leq l$ we have
\begin{align*}
	h_X\big(F(a_i)\big)&=\sum_{k\geq0}\frac{1}{\eta^k}h_X\big(D_k(a_i)\big)h_X(X^{-k})\\
	&=\sum_{k\geq0}\frac{1}{\eta^k}\mu^kD_k\big(h_Y(a_i)\big)\mu^{-k}X^{-k}\\
	&=F\big(h_Y(a_i)\big),
\end{align*}
since $h_X(a)=h(a)=h_Y(a)\in A$ for all $a\in A$. We conclude by noting that 
	\[h_X\big(F(Y)\big)=h_X(X)=\mu X=\mu F(Y)=F(\mu Y)=F\big(h_Y(Y)\big).\]
\end{proof}

\section{Quadratic Poisson Gel'fand-Kirillov problem}

	In this section we give a positive answer to the quadratic Poisson Gel'fand-Kirillov problem (see Introduction) for Poisson algebras satisfying suitable conditions (see Section \ref{PGK}) and some of their quotients (see Section \ref{PGKQ}). This is achieved through repeated use of the characteristic-free Poisson deleting derivation homomorphism constructed in Section 2. In Section 3.1 we give some preliminary results which show that, after deleting the last derivation in an iterated Poisson-Ore extension, moving the last variable in first position does not affect the existences and properties of the needed higher Poisson derivations corresponding to the other variables. This is crucial as it allows for repeated use of the characteristic-free Poisson deleting derivation homomorphism in order to prove the main result of Section 3.2, namely Theorem \ref{abc}. This theorem shows that, under suitable assumptions, there is a Poisson algebra isomorphism between the field of fractions of an iterated Poisson-Ore extension and a Poisson affine space, i.e. the iterated Poisson-Ore extension under consideration satisfies the quadratic Poisson Gel'fand-Kirillov problem. 
 
  Concerning Poisson prime factors of an iterated Poisson-Ore extension $A$, Theorem \ref{abc} tells us that they satisfy the quadratic Poisson Gel'fand-Kirillov problem if the corresponding Poisson prime factors of the Poisson affine space $B$ do (Assertion (2)). In characteristic zero, a Poisson prime factor of a Poisson affine space is always Poisson birationally equivalent to a Poisson affine space over a purely transcendental extension of the base field, see \cite[Theorem 3.3]{GL}. However in prime characteristic this is not clear anymore, and we restrict ourselves to the Poisson prime ideals which are also invariant under the action of a torus $H$. In Section 3.3 we show that, under mild hypotheses, there is actually only finitely many $H$-invariant Poisson prime ideals in a Poisson affine space. Moreover, we explicitly describe all these ideals. As a consequence, the corresponding quotient algebras of $B$ satisfy the quadratic Poisson Gel'fand-Kirillov problem, and so we conclude from Theorem \ref{abc} that all the $H$-invariant Poisson prime  quotient algebras of $A$ satisfy the quadratic Poisson Gel'fand-Kirillov problem.

\subsection{Preliminaries}
	In order to extend the results of the previous section to iterated Poisson-Ore extensions, we need to know the behaviour of a higher Poisson derivation when reordering the variables. This is the objective of the next two lemmas.
\begin{lem}
\label{cas2}
 Let $A$ be a Poisson $\K$-algebra and $R=A[X;\al,\de]_P[Y^{\pm1};\be]_P$ be an iterated Poisson-Ore extension, where $\be(A)\subseteq A$ and $\be(X)=\la X$ for $\la\in \K$.
 \begin{enumerate}
   \item Then $R=A[Y^{\pm1};\be']_P[X;\al',\de']_P$, where $\be'=\be|_{A}$, $\al'|_{A}=\al$, $\de'|_{A}=\de$, $\al'(Y)=-\la Y$ and $\de'(Y)=0$.
   \item If $\de\al=\al\de+\eta\de$ in $A$, then $\de'\al'=\al'\de'+\eta\de'$ in $A[Y^{\pm1};\be]_P$.
   \item Suppose further that $\de$ extends to a higher $(\eta,\al)$-skew Poisson derivation $(D_i)$ on $A$ and that $\be D_i=D_i\be+i\la D_i$ for all $i\geq0$. Then $\de'$ extends to a higher $(\eta,\al')$-skew Poisson derivation $(D_i')$ on $A[Y^{\pm1};\be]_P$ such that the restriction of $D_i'$ to $A$ coincides with $D_i$ for all $i\geq0$, and $D_i'(Y)=0$ for all $i>0$.
		\item Keeping the assumptions of 3 above, we have
		\begin{enumerate}
							\item If $(D_i)$ is iterative, then $(D_i')$ is iterative.
							\item If $(D_i)$ is locally nilpotent, then $(D_i')$ is locally nilpotent.
						\end{enumerate}
 \end{enumerate}
\end{lem}

\begin{proof}
1. Since $\be(A)\subseteq A$ and $\{X,Y\}=-\la XY$ we can switch the variables $X$ and $Y$ in the expression of $R$ as a Poisson-Ore extension over $A$. The new maps we get are those described in 1.

\noindent 2. We only check the equality on a monomial $aY^i\in A[Y^{\pm1}]$ since the derivations involved are $\K$-linear.
\begin{align*}
	\de'\al'(aY^i)&=\de'(\al'(a)Y^i+a\al'(Y^i))\\
	&=\de'(\al'(a)Y^i)+\de'(-i\la aY^i)\\
	&=\de'(\al'(a)Y^i)-i\la(\de'(a)Y^i+\de'(Y^i)a)\\
	&=(\de\al(a)-i\la\de(a))Y^i\\
	&=\big(\al\de(a)+\eta\de(a)-i\la\de(a)\big)Y^i\\
	&=(\al'\de'+\eta\de')(aY^i).
\end{align*}

\noindent 3. Define a sequence of $\K$-linear maps $D_i': A[Y^{\pm1};\be']_P\rightarrow A[Y^{\pm1};\be']_P$ for all $i\geq0$ by
	\[D_i'\Big(\sum\limits_{j=-m}^{m}a_jY^j\Big)=\sum\limits_{j=-m}^{m}D_i(a_j)Y^j.\]
We check that $(D_i')$ is a higher $(\eta,\al')$-skew Poisson derivation on $A[Y^{\pm1};\be']_P$ satisfying all conditions of 3.
	 First, it is clear that $D_i'(a)=D_i(a)$ for all $a\in A$. Moreover $D_i'(Y)=D_i(1)Y=0$ for $i>0$ and $D_0'=\id$ on $A[Y^{\pm1};\be']_P$. The following computation shows that $\de'$ extends to $(D_i')$:
	\[D_1'\Big(\sum\limits_{j=-m}^{m}a_jY^j\Big)=\sum\limits_{j=-m}^{m}D_1(a_j)Y^j=\sum\limits_{j=-m}^{m}\de(a_j)Y^j=\de'\Big(\sum\limits_{j=-m}^{m}a_jY^j\Big).\]

It just remains to establish Axioms (A1), (A2) and (A3) of Definition \ref{hd} on monomials of $A[Y^{\pm1}]$ (since the Poisson bracket is $\K$-bilinear and the $D_i'$ and the $D_i$ are $\K$-linear maps).

First, for all $a,b\in A$ and all $i,j\in\Z$:
\begin{align*}
	D_n'\big((aY^i)(bY^j)\big)&=D_n(ab)Y^{i+j}\\
	&=\sum_{k=0}^{n}D_k(a)D_{n-k}(b)Y^{i+j}\\
	&=\sum_{k=0}^{n}D_k'(aY^i)D_{n-k}'(bY^j).
\end{align*}
Hence Axiom (A1) is proved. Next
\begin{align*}
D_n'(\{aY^i,bY^j\})&=D_n'\big[\big(\{a,b\}+i\be'(b)a-j\be'(a)b\big)Y^{i+j}\big]\\
&=\big[D_n(\{a,b\})+iD_n(\be(b)a)-jD_n(\be(a)b)\big]Y^{i+j}\\
&=\sum\limits_{k=0}^{n}\Big[\{D_k(a),D_{n-k}(b)\}+k\big(\al D_{n-k}(a)D_k(b)-D_k(a)\al D_{n-k}(b)\big)\Big]Y^{i+j}\\
&\qquad+i\sum\limits_{k=0}^{n}D_{n-k}(a)D_k\be(b)Y^{i+j}\\
&\qquad-j\sum\limits_{k=0}^{n}D_{n-k}(b)D_k\be(a)Y^{i+j},
\end{align*}
whereas
\begin{align*}
&\quad\sum\limits_{k=0}^{n}\{D_k'(aY^i),D_{n-k}'(bY^j)\}+k\big(\al'D_{n-k}'(aY^i)D_k'(bY^j)-D_k'(aY^i)\al'D'_{n-k}(bY^j)\big)\\
&=\sum\limits_{k=0}^{n}\Big(\{D_k(a),D_{n-k}(b)\}+iD_k(a)\be'D_{n-k}(b)-j\be'D_k(a)D_{n-k}(b)\Big)Y^{i+j}\\
&\qquad+\sum\limits_{k=0}^{n}kD_k(b)\big(\al D_{n-k}(a)Y^i+D_{n-k}(a)\al'(Y^i)\big)Y^j\\
&\qquad-\sum\limits_{k=0}^{n}kD_k(a)\big(\al D_{n-k}(b)Y^j+D_{n-k}(b)\al'(Y^j)\big)Y^i\\
&=\sum\limits_{k=0}^{n}\Big(\{D_k(a),D_{n-k}(b)\}+iD_{k}(a)\be D_{n-k}(b)-j\be D_k(a)D_{n-k}(b)\Big)Y^{i+j}\\
&\qquad+\sum\limits_{k=0}^{n}kD_k(b)\big(\al D_{n-k}(a)-i\la D_{n-k}(a)\big)Y^{i+j}\\
&\qquad-\sum\limits_{k=0}^{n}kD_k(a)\big(\al D_{n-k}(b)-j\la D_{n-k}(b)\big)Y^{i+j}\\
&=\sum\limits_{k=0}^{n}\Big(\{D_k(a),D_{n-k}(b)\}+k\big(\al D_{n-k}(a)D_k(b)-D_k(a)\al D_{n-k}(b)\big)\Big)Y^{i+j}\\
&\qquad+i\sum\limits_{k=0}^{n}D_{k}(a)\be D_{n-k}(b)Y^{i+j}-i\la\sum\limits_{k=0}^{n}kD_{n-k}(a)D_k(b)Y^{i+j}\\
&\qquad-j\sum\limits_{k=0}^{n}D_{n-k}(b)\big(\be D_k(a)-k\la D_k(a)\big)Y^{i+j}\\
&=\sum\limits_{k=0}^{n}\Big(\{D_k(a),D_{n-k}(b)\}+k\big(\al D_{n-k}(a)D_k(b)-D_k(a)\al D_{n-k}(b)\big)\Big)Y^{i+j}\\
&\qquad+i\sum\limits_{k=0}^{n}D_{n-k}(a)\big(\be D_{k}(b)-k\la D_k(b)\big)Y^{i+j}\\
&\qquad-j\sum\limits_{k=0}^{n}D_{n-k}(b)\big(\be D_k(a)-k\la D_k(a)\big)Y^{i+j}.
\end{align*}
(In the last step of this computation we used a change of variable $k'=n-k$ in the second sum). Since $\be D_k-\la kD_k=D_k\be$ for all $k\geq0$, Axiom (A2) is established.
And finally, we get Axiom (A3) by computing:
\begin{align*}
(\al'D_i'+i\eta D_i')(aY^l)&=\al'\big(D_i(a)Y^l\big)+i\eta D_i(a)Y^l\\
&=\big(\al D_i(a)-\la lD_i(a)+i\eta D_i(a)\big)Y^l\\
&=\big(D_i\al(a)-\la lD_i(a)\big)Y^l\\
&=D_i(\al(a)-\la la)Y^l\\
&=D_i'\big((\al(a)-\la la)Y^l\big)\\
&=D_i'\big(\al'(aY^l)\big),
\end{align*}
for all $i\geq 0$ and $l\in\Z$.

\noindent 4(a).	If $(D_i)$ is iterative on $A$, then
	\[D_i'D_j'(aY^l)=D_i'(D_j(a)Y^l)=D_iD_j(a)Y^l=\binom{i+j}{j}D_{i+j}(a)Y^l=\binom{i+j}{j}D_{i+j}'(aY^l)\]
for all $a\in A$, $l\in\Z$ and $i,j\geq0$. Hence $(D_i')$ is iterative on $A[Y^{\pm1};\be']_P$.\\
4(b).	Suppose that $(D_i)$ is locally nilpotent on $A$. Using Lemma \ref{ln} we only need to check that $(D_i')$ is locally nilpotent on a set of generators of  $A[Y^{\pm1}]$. We take $A\cup\{Y^{\pm1}\}$. For all $a\in A$ and $i\geq0$ we have $D_i'(a)=D_i(a)$, so that $(D_i')^n(a)=0$ for $n>>0$. Moreover $D_i'(Y)=0$ (which implies $D_i'(Y^{-1})=0$) for all $i>0$. The result is shown.
\end{proof}

	Lemma \ref{cas2} can be generalised as follows.

\begin{lem}
\label{casn}
Let $A$ be a Poisson $\K$-algebra and set $R=A[X_1;\al_1,\de_1]_P\cdots[X_n;\al_n,\de_n]_P[Y^{\pm1};\be]_P$, where $\be(A)\subseteq A$ and $\be(X_i)=\la_iX_i$ with $\la_i\in \K$ for all $1\leq i\leq n$. Let \\
$R_j=A[X_1;\al_1,\de_1]_P\cdots[X_j;\al_j,\de_j]_P$ for $j=1,\dots,n$ and $R_0=A$.
\begin{enumerate}
    \item Then $R=A[Y^{\pm1};\be']_P[X_1;\al_1',\de_1']_P\cdots[X_n;\al_n',\de_n']_P$, where $\be'=\be|_A$, $\al_{i}'|_{R_j}=\al_i$, $\de_i'|_{R_j}=\de_i$, $\al_i'(Y)=-\la_i Y$ and $\de_i'(Y)=0$ for all $i=1,\dots,n$ and $j=0,\dots,i-1$.
    \item Set $R_j'=A[Y^{\pm1};\be']_P[X_1;\al_1',\de_1']_P\cdots[X_j;\al_j',\de_j']_P$. For all $i$, if $\de_i\al_i=\al_i\de_i+\eta_i\de_i$ on $R_{i-1}$, then $\de_i'\al_i'=\al_i'\de_i'+\eta_i\de_i'$ on $R_{i-1}'$.
    \item Suppose that each $\de_i$ extends to a higher $(\eta_i,\al_i)$-skew Poisson derivation $(D_{i,k})_{k=0}^\infty$, and that $\be D_{i,k}=D_{i,k}\be+k\la_iD_{i,k}$ on $R_{i-1}$ for all $i$ and $k$. Then each $\de_i'$ extends to a higher $(\eta_i,\al_i')$-skew Poisson derivation $(D_{i,k}')_{k=0}^\infty$ on $R'_{i-1}$, where $D_{i,k}'$ coincides with $D_{i,k}$ on $R_j$, for $j<i$, and $D_{i,k}'(Y)=0$ for $k>0$.
    \item Keeping the assumptions of 3 above, we have
     \begin{enumerate}
        		\item If $(D_{i,k})_{k=0}^\infty$ is iterative, then $(D_{i,k}')_{k=0}^\infty$ is iterative.
        		\item If $(D_{i,k})_{k=0}^\infty$ is locally nilpotent, then $(D_{i,k}')_{k=0}^\infty$ is locally nilpotent.
        	\end{enumerate}
\end{enumerate}
\end{lem}

\begin{proof}
Easy induction left to the reader.
\end{proof}

\subsection{Quadratic Poisson Gel'fand-Kirillov problem}
\label{PGK}

	The theorem below gives conditions under which (a quotient of) a suitable iterated Poisson-Ore extension is Poisson birationally equivalent to (a quotient of) a Poisson affine space. Recall that a \textit{Poisson prime ideal} $P$ of a Poisson algebra $A$ is a prime ideal which is also a Poisson ideal, i.e. such that $\{a,u\}\in P$ for all $a\in A$ and $u\in P$.	An ideal $I$ of a $\K$-algebra supporting a torus $H$-action by Poisson automorphisms is said \textit{$H$-invariant} if $H(I)=I$.

\begin{thm}
\label{abc}
Let $A=\K[X_1][X_2;\alpha_2,\delta_2]_P\cdots[X_n;\alpha_n,\delta_n]_P$ be an iterated Poisson-Ore extension such that each derivation $\de_i$ extends to an iterative, locally nilpotent higher $(\eta_i,\al_i)$-skew Poisson derivation $(D_{i,k})_{k=0}^{\infty}$ on 	\[A_{i-1}=\K[X_1][X_2;\alpha_2,\delta_2]_P\cdots[X_{i-1};\alpha_{i-1},\delta_{i-1}]_P,\]
where each $\eta_i$ is a nonzero scalar. Suppose furthermore that for all $1\leq j<i\leq n$ there exists $\la_{ij}\in\K$ such that $\al_i(X_j)=\la_{ij}X_j$, and $\al_iD_{j,k}=D_{j,k}\al_i+k\la_{ij}D_{j,k}$ for all $k\geq0$. Let $\boldsymbol{\la}=(\la_{ij})$ be the skew-symmetric matrix in $M_n(\K)$ whose coefficients below the diagonal are the above scalars. Then:

	(1) There exists a Poisson algebra isomorphism between $\F A$ and $\K_{\boldsymbol{\la}}(Y_1,\ldots,Y_n)$.
	
	(2) For any Poisson prime ideal $P$ in $A$, there exists a Poisson prime ideal $Q$ in $B=\K_{\boldsymbol{\la}}[Y_1,\ldots,Y_n]$ such that the fields $\F A/P$ and $\F B/Q$ are isomorphic as Poisson algebras.
	
	(3) Assume that the torus $H=(\K^{\times})^r$ is acting rationally by Poisson automorphisms on $A$ such that each $X_i$ is an $H$-eigenvector, and $B$ is endowed with the induced $H$-action (for all $h\in H$ and all $1\leq i\leq n$ there exists $\mu_i\in\K^{\times}$ such that $h(X_i)=\mu_iX_i$; then the action of $h$ on the generator $Y_i$ of $B$ is given by $h(Y_i)=\mu_iY_i$). Moreover we suppose that $h\big(D_{i,k}(X_j)\big)=\mu_i^kD_{i,k}\big(h(X_j)\big)$ for all $1\leq j<i\leq n$ and $k\geq0$. Then, for any $H$-invariant Poisson prime ideal $P$ in $A$, there exists an $H$-invariant Poisson prime ideal $Q$ in $B=\K_{\boldsymbol{\la}}[Y_1,\ldots,Y_n]$ such that the fields $\F A/P$ and $\F B/Q$ are isomorphic as Poisson algebras.
\end{thm}
\begin{proof}

	We prove these results all together by three inductions: first on $n$, second on the number $d$ of indices $i$ for which $\de_i\neq 0$ and finally on the maximum index $t$ for which $\de_t\neq 0$ (this last induction being downward). If $d=0$ then set $t:=n+1$.

	If $n=1$ or $t=n+1$ the result is shown. Indeed if $n=1$, $\F(\K[X])=\K(X)$ and if $t=n+1$, then $d=0$ and $A=\K[X_1][X_2;\al_2]_P\cdots[X_n;\al_n]_P=\K_{\bo\la}[X_1,\ldots,X_n]\cong B$. So we can assume that $n\geq2$ and $t\leq n$.

	Let $P$ be a Poisson prime ideal in $A$. Assertion (1) is satisfied when $P=Q=0$ in (2). Assertions (2) and (3) are shown simultaneously. The proof splits in three cases: first if $X_n\in P$, next if $X_n\notin P$ and $t=n$, and finally if $X_n\notin P$ and $t<n$; each case will be solved by a different induction. Note that, for all $1\leq i\leq n$, the $H$-actions on $A$ and $B$ induce, by restriction, $H$-actions on the subalgebras $A_i$ and $B_i:=\K_{\bo\la_i}[X_1,\dots,X_i]$, where $\bo\la_i$ is the upper left $i\times i$ submatrix of $\bo\la$. When $P$ is an $H$-invariant ideal of $A$ we also consider the induced action of $H$ on $A/P$. These actions are all rational actions by Poisson automorphisms, such that the generators of the algebras considered are $H$-eigenvectors.\\

\textbf{First case:} $X_n\in P$. Consider the Poisson algebra homomorphism $\Phi : A_{n-1} \rightarrow A/P$ defined by $\Phi(X_i)=\ov{X_i}$ for all $i<n$. Since $\Phi$ is surjective, there exists a Poisson prime ideal $P'=\ker(\Phi)$ in $A_{n-1}$ such that $A/P\cong A_{n-1}/P'$. Moreover it is clear that $P'$ is $H$-invariant if $P$ is $H$-invariant since the diagram of Figure \ref{diag4} is commutative for all $h\in H$.
\begin{figure}[!ht]
	\centering
		\begin{tikzpicture}
	\node (a) at (0,0){$A_{n-1}$};
	\node (b) at (4,0){$A/P$};
	\node (c) at (0,-2){$A_{n-1}$};
	\node (d) at (4,-2){$A/P$};
	\node (aa) at (2,0.2){$\Phi$};
	\node (bb) at (-0.4,-1){$h$};
	\node (cc) at (2,-2.2){$\Phi$};
	\node (dd) at (4.3,-1){$h$};
	\draw [->] (a) to (b);\draw [->] (c) to (d);
	\draw [->] (a) to (c);\draw [->] (b) to (d);
		\end{tikzpicture}
		\caption{}
	\label{diag4}
\end{figure}
By the first induction (on $n$), there exists an ($H$-invariant if $P$ is $H$-invariant) Poisson prime ideal $Q'$ in the algebra $B_{n-1}$ such that $\F A_{n-1}/P'\cong \F B_{n-1}/Q'$. Observe that $Q=Q'+BY_n$ is an ($H$-invariant if $P$ is $H$-invariant) Poisson prime ideal in $B$ such that $B_{n-1}/Q'\cong B/Q$. Thus $\F A/P \cong \F B/Q$.\\

\textbf{Second case:} $X_n\notin P$ and $t=n$. So $\de_n\neq0$. Set $A'=A_{n-1}[Y;\al_n]_P$. Since $\de_n$ extends to an iterative, locally nilpotent higher $(\eta_n,\al_n)$-skew Poisson derivation $(D_{n,k})_{k=0}^{\infty}$ on $A_{n-1}$, it follows from Proposition \ref{iso} that $A_{n-1}[X_n^{\pm1};\al_n,\de_n]_P\cong A_{n-1}[Y^{\pm1};\al_n]_P$ and so $A[X_n^{-1}]\cong A'[Y^{-1}]$. Thus there exists a Poisson prime ideal $P'=P[X_n^{-1}]\cap A'$ in $A'$ such that $\F A/P \cong \F A'/P'$, where $P'=0$ if $P=0$. As in Section \ref{H action}, the action of $H$ on $A_{n-1}$ extends to $A_{n-1}[Y^{\pm1};\al_n]_P$ by setting $h(Y)=\mu_nY$ (where $\mu_n\in\K^{\times}$ is defined by $h(X_n)=\mu_nX_n$). Then, if the ideal $P$ is $H$-invariant, the ideal $P'$ is $H$-invariant since the Poisson isomorphism $A_{n-1}[X_n^{\pm1};\al_n,\de_n]_P\cong A_{n-1}[Y^{\pm1};\al_n]_P$ commutes with all $h\in H$ (choose $\{X_1,\dots,X_{n-1}\}$ for a generating set of $A_{n-1}$ and apply Lemma \ref{com} with $A_{n-1}$ as coefficient ring). Finally, the number of nonzero maps among $\de_2,\ldots,\de_{n-1}$ is $d-1$, so the induction step (on $d$) gives the result for $\F A'/P'$ and so for $\F A/P$.\\

\textbf{Third case:} $X_n\notin P$ and $t<n$. Thus $\de_n=0$. By Lemma \ref{casn} we can write $A[X_n^{-1}]$ in the form 
\begin{align*}
	A[X_n^{-1}]=\K[X_1][X_n^{\pm1};\al_n']_P[X_2;\alpha_2',\delta_2']_P\cdots[X_{n-1};\alpha_{n-1}',\delta_{n-1}']_P 
\end{align*}
where $\al'_i(X_j)=\la_{ij}X_j$ for $j<i$ and $j=n$, and each $\de_i'$ extends to an iterative, locally nilpotent higher $(\eta_i,\al_i')$-skew Poisson derivation $(D_{i,k}')_{k=0}^{\infty}$ on
	\[A_{i-1}':=\K[X_1][X_n^{\pm1};\al_n']_P[X_2;\al_2',\de_2']_P\cdots[X_{i-1};\al_{i-1}',\de_{i-1}']_P.\]
It is clear that we have $\al'_iD'_{j,k}=D'_{j,k}\al'_i+k\la_{ij}D'_{j,k}$ and $h\big(D'_{i,k}(X_j)\big)=\mu_i^kD'_{i,k}\big(h(X_j)\big)$ for all $1\leq j<i\leq n$, all $k\geq0$ and all $h\in H$, since by Lemma \ref{casn} we have:
\[D'_{i,k}(X_j)=\left\{
    \begin{array}{lll}
         D_{i,k}(X_j) \qquad &j<i\text{ and }k\geq0,\\
         X_n &j=n\text{ and }k=0,\\
         0 & j=n\text{ and }k\geq1.
    \end{array}\right.\]
	We can now use the induction hypothesis since the derivation $\de_t'$ is nonzero ($\de_t'$ restrict to $\de_t$) and occurs in position $t+1$ in the list $0,0,\de_2',\ldots,\de_{n-1}'$. And thus the induction on $t$ allows to conclude.
\end{proof}

	By Example \ref{exgl}, when $\car\K=0$, the hypotheses of \cite[Theorem 3.9]{GL} imply those of our Theorem \ref{abc} (except Assertion 3). Hence Assertions 1 and 2 of our Theorem \ref{abc} generalise \cite[Theorem 3.9]{GL} to any characteristic.
	
\subsection{Quadratic Poisson Gel'fand-Kirillov problem for quotients by $H$-invariant Poisson prime ideals}
\label{PGKQ}

	Assertions 2 and 3 of Theorem \ref{abc} tells us that $H$-invariant Poisson prime factors of the iterated Poisson-Ore extensions under consideration are Poisson birationally isomorphic to $H$-invariant Poisson prime factors of Poisson affine spaces. In this section, we go one step further and prove that these factor algebras satisfy the quadratic Poisson Gel'fand-Kirillov problem under some mild assumptions on the torus action (Hypothesis \ref{hyp}) and the base field $\K$. 

	More precisely, set $[\mspace{-2 mu} [1,n] \mspace{-2 mu} ]:=\{1,\dots,n\}$ and $E:=\mathscr{P}([\mspace{-2 mu} [1,n] \mspace{-2 mu} ])$, the set of subsets of $[\mspace{-2 mu} [1,n] \mspace{-2 mu} ]$. The key is to show that, under a suitable $H$-action, the only $H$-invariant Poisson prime ideals of a Poisson affine space $B=\K_{(\la_{ij})}[Y_1,\dots,Y_n]$ are the ideals $J_w:=<Y_i\ |\ i\in w>$, where $w\in E$. This is achieved in Section \ref{hinv}. As a consequence the $H$-invariant Poisson prime factors of $B$ are again Poisson affine spaces over $\K$, and therefore satisfy the quadratic Poisson Gel'fand-Kirillov problem. We conclude from Theorem \ref{abc} that $H$-invariant Poisson prime factors of the iterated Poisson-Ore extensions considered also satisfy the quadratic Poisson Gel'fand-Kirillov problem.
	
	From now on, we require that the field $\K$ contains at least one element which is not a root of unity, or that $\K$ is algebraically closed.

\subsubsection{Assumptions on the $H$-action}
\label{ass}

In this section we recall some classical facts on rational torus action and present the hypotheses we need in the following section.

	Let $r>0$. Suppose that the torus $H=(\K^{\times})^r$ is acting rationally by Poisson automorphisms on the iterated Poisson-Ore extension $A=\K[X_1][X_2;\al_2,\de_2]_P\cdots[X_n;\al_n,\de_n]_P$ such that each $X_i$ is an $H$-eigenvector, and suppose that there exist scalars $\la_{ij}$ for all $1\leq j<i\leq n$ such that $\al_i(X_j)=\la_{ij}X_j$. The rational character group $X(H)$ of $H$ is identified with the group $\Z^r$ via the bijection
\begin{align*}
	\Z^r\qquad &\longrightarrow \qquad\qquad X(H)\\
	\underline{x}=(x_1,\dots,x_r)&\longmapsto  \Big((h_1,\dots,h_r)\longmapsto h_1^{x_1}\cdots h_r^{x_r}\Big).
\end{align*}

	Since $H$ is a torus, the rationality of the action means that $A$ is the direct sum of its $H$-eigenspaces, and the corresponding eigenvalues are rational characters of $H$ (i.e. they are homomorphisms of algebraic varieties $(\K^{\times})^r\rightarrow \K^{\times}$), see \cite[Theorem II.2.7]{BG}. For $1\leq i\leq n$ we denote by $\underline{f}_i\in\Z^r$ the character associated to $X_i$. For $\ul{\mu}=(\mu_1,\dots,\mu_r)\in\Z^r$ and $\ul{\nu}=(\nu_1,\dots,\nu_r)\in\Z^r$, we set $(\ul{\mu}|\ul{\nu}):=\sum_{i=1}^r\mu_i\nu_i$.

	In the following we restrict our attention on Poisson algebras satisfying Hypothesis \ref{hyp}. In Section \ref{examples} we will present many examples of such algebras.
\begin{hypo}
\label{hyp}
For all $1\leq i\leq n$, there exists $\ul{\gam}_i\in\Z^r$ such that 
\begin{itemize}
	\item $\la_{ij}=(\ul{\gam}_i|\ul{f}_j)$ for all $1\leq j<i$;
	\item $\rho_i:=(\ul{\gam}_i|\ul{f}_i)\in\K^{\times}$.
\end{itemize}
\end{hypo}

	Form the skew-symmetric matrix $\boldsymbol{\la}\in M_n(\K)$ whose coefficients below the diagonal are the $\la_{ij}$ and, as in Assertion 3 of Theorem \ref{abc}, endow $B=\K_{\bo\la}[Y_1,\dots,Y_n]$ with the rational $H$-action by Poisson automorphisms induced by the $H$-action on $A$. Note that for all $1\leq i\leq n$, the indeterminate $Y_i$ is an $H$-eigenvector with associated character $\underline{f}_i\in\Z^r$.

\subsubsection{$H$-invariant ideals in Poisson affine spaces}
\label{hinv}

	For $w\in E$ we set $\ov{w}:=[\mspace{-2 mu} [1,n] \mspace{-2 mu} ]\setminus w$. Assume $\ov{w}\neq\emptyset$. We denote by $S_w$ the multiplicative set of $B/J_w$ generated by the $Y_i+J_w$ for $i\in \ov{w}$, and consider the algebra
	\[T=(B/J_w)S_w^{-1}.\]
We set $\ov{w}:=\{l_1,\dots,l_s\}$, where $1\leq l_1<\dots<l_s\leq n$ and $s\in\{1,\dots,n\}$. For all $i\in\{1,\dots,s\}$, set $U_i:=Y_{l_i}+J_w$ and for all $1\leq j<i\leq n$, set $\la'_{ij}:=\la_{l_i l_j}$.
Then $T$ is the Poisson torus $T=\K_{(\la'_{ij})}[U_1^{\pm1},\dots,U_s^{\pm1}]$, where $(\la'_{ij})$ is the skew-symmetric matrix whose coefficients under the diagonal are the scalars $\la'_{ij}$ defined previously.
	
	Since the ideal $J_w$ and the multiplicative set $S_w$ are generated by $H$-eigenvectors, the torus $H$ is acting rationally by Poisson automorphisms on $T$ and for all $1\leq i\leq s$ the indeterminate $U_i$ is an $H$-eigenvector with associated character $\ul{u}_i:=\ul{f}_{l_i}$. Moreover for all $i\in\{1,\dots,s\}$, we set $\ul{\gam}'_i:=\ul{\gam}_{l_i}$ and $\rho'_i:=\rho_{l_i}$. Thus we have $\la'_{ij}=(\ul{\gam}'_i|\ul{u}_j)$ for all $1\leq j<i\leq s$ and $\rho'_i=(\ul{\gam}'_i|\ul{u}_i)\in\K^{\times}$ for all $1\leq i\leq s$.
	
\begin{lem}
\label{key}
	Let $(m_1,\dots,m_s)\in\Z^s\setminus(0,\dots,0)$ and suppose that $U:=U_1^{m_1}\cdots U_s^{m_s}$ is a Poisson central element in $T$. Then there exists $h \in H$ such that $h(U)=\varepsilon U$ with $\varepsilon\in\K\setminus\{0,1\}$.
\end{lem}

\begin{proof} We can assume that $m_s$ is nonzero. Otherwise replace $s$ by the largest $i$ such that $m_i\neq0$ in the following. Start by noting that $U\in Z_p(T)$ implies that $0=\{U,U_s\}=\big(\sum_{i<s}m_i\la'_{si}\big)UU_s$, i.e. $\sum_{i<s}m_i\la'_{si}=0$. 

	Let $i<s$. Set $\ul{\gam}'_s:=(\mu_1,\dots,\mu_r)\in\Z^r$. Thus we have $\la'_{si}=\sum_{j=1}^r \mu_j\nu_j$ with the notation $\ul{u}_i:=(\nu_1,\dots,\nu_r)\in\Z^r$.
	Let $q\in\K^{\times}$ and set $h_s:=(q^{\mu_1},\dots,q^{\mu_r})\in H$. Still identifying $X(H)$ with $\Z^r$, we have 
\begin{align*}
	h_s(U_i)=\ul{u}_i(h_s)U_i=(q^{\mu_1})^{\nu_1}\cdots(q^{\mu_r})^{\nu_r}U_i=q^{\la'_{si}}U_i
\end{align*}
for all $i<s$, and
	\[h_s(U_s)=\ul{u}_s(h_s)U_s=q^{(\ul{\gam}'_s|\ul{u}_s)}U_s=q^{\rho'_s}U_s.\]	
	So $h_s(U)=q^{\sum_{i<s}m_i\la'_{si}}q^{\rho'_s m_s}U=q^{\rho'_s m_s}U$. By the assumptions on the ground field made at the beginning of Section \ref{PGKQ}, we can choose $q$ such that $q^{\rho'_s m_s}\neq1$, and the result is shown.
\end{proof}

The following result comes from \cite{Vancliff}. The author was working over the base field $\C$ but the result is still true over an arbitrary infinite base field.

\begin{lem}[Vancliff]
\label{Vancliff}
	If $I$ is a Poisson ideal of $T$, then $I$ is generated by its intersection with the Poisson center of $T$.
\end{lem}

\begin{prop}
\label{torus}
	If $I$ is an $H$-invariant Poisson prime ideal of $T$, then $I=\{0\}$.
\end{prop}

\begin{proof}
	Suppose $I\neq\{0\}$. By Lemma \ref{Vancliff}, there exists a nonzero Poisson central element $V\in I$. Write $V=\la_1U^{m_1}+\cdots+\la_kU^{m_k}$ with $m_1,\dots,m_k\in\Z^s$ pairwise distinct, $\la_1,\dots,\la_k\in\K^{\times}$ and $k>0$. Suppose that $V$ is chosen in such way that $k$ is minimal. If $k=1$, then $V$ is invertible and $I=T$, a contradiction, thus we suppose $k>1$.
	
	The monomials $U^{m_1},\dots,U^{m_k}$ are Poisson central, invertible and $U^{m_k}(U^{m_1})^{-1}=U^m$ with $m=m_k-m_1\in\Z^s\setminus(0\dots,0)$. Thus by Lemma \ref{key} there exists $h\in H$ such that $h(U^{m_k}(U^{m_1})^{-1})=\varepsilon U^{m_k}(U^{m_1})^{-1}$ with $\varepsilon\in\K\setminus\{0,1\}$. Since $U_1,\dots,U_s$ are $h$-eigenvectors, then so are $U^{m_1},\dots,U^{m_k}$ and we can write $h(U^{m_i})=\nu_iU^{m_i}$ with $\nu_i\in \K^{\times}$ for all $1\leq i\leq k$. Consider now the Poisson central element $W=V-\nu_1^{-1}h(V)\in I$. We have
\begin{align*}
	W=\sum_{i=1}^k\la_i(1-\nu_i\nu_1^{-1})U^{m_i}=\sum_{i=2}^k\la_i(1-\nu_i\nu_1^{-1})U^{m_i}.
\end{align*}
Since $\varepsilon U^{m_k}(U^{m_1})^{-1}=h(U^{m_k}(U^{m_1})^{-1})=\nu_k\nu_1^{-1}U^{m_k}(U^{m_1})^{-1}$ we have $\nu_k\nu_1^{-1}\neq1$ and so $W\neq0$. Thus $W$ is a nonzero Poisson central element of $I$ which can be written as a sum of at most $k-1$ monomials. This contradicts the choice of $k$.
\end{proof}

\begin{nota}
The \textit{Poisson prime spectrum of $B$}, denoted P$\Sp(B)$, is the subset of Poisson ideals in $\Sp(B)$. For all $w\in E$ we defined a subset of P$\Sp(B)$ by
\begin{align*}
	\rm{P}\Sp_\textit{w}(\textit{B})=\Big\{I\in\rm{P}\Sp(\textit{B})\ |\ I\cap \{Y_1,\dots,Y_n\}=\{Y_i\ |\ i\in \textit{w}\}\Big\}.
\end{align*}
These subsets form a partition of P$\Sp(B)$.
\end{nota}

\begin{prop}
\label{w}
The only $H$-invariant Poisson prime ideals of $B$ are the ideals
	\[J_w=<Y_i\ |\ i\in w>\]
 for all $w$ of $E$.
\end{prop}

\begin{proof}
	Let $I$ be an $H$-invariant Poisson prime ideal of $B$. There exists $w\in E$ such that $I \in\rm{P}\Sp_\textit{w}(\textit{B})$. If $w=\{1,\dots,n\}$, then $J_w$ is a maximal ideal and thus $I=J_w$.

	Suppose $w\neq\{1,\dots,n\}$. Then $J_w\subset I$ and $I/J_w$ is a Poisson prime ideal of $B/J_w$ which does not intersect the multiplicative set $S_w$. Thus $P=(I/J_w)S_w^{-1}$ is a Poisson prime ideal of the Poisson torus $T=(B/J_w)S_w^{-1}$. Since $I$ is $H$-invariant and all elements of $S_w$ are $H$-eigenvectors, the ideal $P$ is $H$-invariant. Proposition \ref{torus} implies $P=\{0\}$ and so $I=J_w$, as desired.
\end{proof}

Combining Proposition \ref{w} and Theorem \ref{abc} we obtain the main result of this section.
	
\begin{thm}
\label{all hypo}
Let $A$ be an iterated Poisson-Ore extension satisfying all the hypotheses of Theorem \ref{abc}. Assume that Hypothesis \ref{hyp} is satisfied. Then, for any $H$-invariant Poisson prime ideal $P$ of $A$, the field of fractions $\F A/P$ is Poisson isomorphic to a Poisson affine field $\K_{\bo\la'}(Z_1,\dots,Z_m)$, where $m\leq n$ and $\bo\la'\in M_m(\K)$ is a skew-symmetric matrix.
\end{thm}
	
\begin{proof}
By Theorem \ref{abc} we have $\F A/P\cong\F B/Q$ where $B=\K_{\bo\la}[Y_1,\dots,Y_n]$ and $Q$ is an $H$-invariant Poisson prime ideal of $B$. By Proposition \ref{w} there exists $w\in E$ such that $Q=J_w$. Then $B/Q=\K_{\bo\la'}[\ov{Y_i}\ |\ i\notin w]$, where $\bo\la'$ is the skew-symmetric submatrix of $\bo\la$ obtained by deleting rows and columns indexed by $i\in w$. The result follows.
\end{proof}
	
	Theorem \ref{all hypo} is new even in characteristic zero. In the following section, we prove a result that shows that the hypotheses of Theorem \ref{all hypo} are satisfied for a large class of Poisson polynomial algebras.

\section{Quadratic Poisson Gel'fand-Kirillov problem and semiclassical limits}
\label{examples}

	In this section we give examples of Poisson $\K$-algebras satisfying the hypotheses of  Theorem \ref{all hypo}, so that they satisfy the quadratic Poisson  Gel'fand-Kirillov problem described in the introduction. These Poisson $\K$-algebras actually arise as semiclassical limits of quantum algebras described in \cite[Section 5]{HH}. In order to prove a transfer result, one needs to address the existence of higher Poisson derivations on the Poisson algebras considered. Contrary to the characteristic zero case, higher derivations in prime characteristic seem not to be well understood. However, we can ensure the existence in arbitrary characteristic by the semiclassical limit process. This mainly relies on the fact that we can define a quantum analogue of a higher derivation independently of the characteristic of the base field, as long as the deformation parameter is transcendental over the base field (this is always the case in the setting of the semiclassical limit process). Our transfer result (Theorem \ref{fin}) states, in particular, that this quantum analogue of a higher derivation induces a Poisson higher derivation on the semiclassical limit. More generally Theorem \ref{fin} gives conditions on a quantum algebra under which its semiclassical limit satisfies the quadratic Poisson Gel'fand-Kirillov problem. In Sections 4.2 and 4.3 we illustrate our results in the case of (coordinate rings of) Poisson matrix varieties, viewed as the semiclassical limits of (coordinate rings of) quantum matrices, and some of their quotients, namely (coordinate rings of) Poisson determinantal varieties.
	
	We continue to assume that the ground field $\K$ contains at least an element which is not a root of unity, or that $\K$ is algebraically closed.

\subsection{Semiclassical limit process and existence of higher Poisson derivation}

	We begin by recalling the semiclassical limit process. Let $\mathcal{R}$ be a (non necessarily commutative) integral domain over $\K[t^{\pm1}]$. The element $(t-1)$ is central and we denote by $(t-1)\mathcal{R}$ the (left and right) ideal generated by $(t-1)$. For $r,s\in\mathcal{R}$, we set $[r,s]:=rs-sr$. Assume that the algebra $R=\mathcal{R}/(t-1)\mathcal{R}$ is commutative. For convenience, the image of an element $r\in\mathcal{R}$ in the quotient algebra will be denoted alternatively by $\ov{r}$, $r+(t-1)\mathcal{R}$ or $r|_{t=1}$. Since $R$ is commutative, we have $[r,s]\in (t-1)\mathcal{R}$, and so there exists a unique element $\gam(r,s)\in\mathcal{R}$ such that $[r,s]=(t-1)\gam(r,s)=\gam(r,s) (t-1)$. We will often use the following notation for the element $\gam(r,s)$: 
	$$\frac{[r,s]}{t-1}:=\gam(r,s).$$
It is well known that one defines a Poisson bracket on $R$ by setting 
	\[\{\ov{r},\ov{s}\}:=\frac{[r,s]}{t-1}\Big|_{t=1}\]
for all $r,s\in\mathcal{R}$. The commutative algebra $R$ endowed with this Poisson bracket is called the \textit{semiclassical limit} of $\mathcal{R}$ at $t-1$. For more details on semiclassical limit we refer the reader to \cite[III.5.4]{BG} or \cite[Section 1.1.3]{Dum} for instance.

Before going any further, we need to recall the notion of $q$-integers and $q$-binomial coefficients, where $q$ is a nonzero non-root-of-unity element of $\K[t^{\pm 1}]$.  Our conventions are as follows. For all $i\geq k\geq0$ we set
\begin{align*}
	(i)_q&=q^{i-1}+q^{i-2}+\cdots+1,\\
	(i)!_q&=(i)_q(i-1)_q\cdots(1)_q,\\
	\binom{i}{k}_q&=\frac{(i)!_q}{(i-k)!_q(k)!_q}.
\end{align*}
By convention $(0)!_q=1$. In the following, we will use $q$-integers in the case where $q=t^{\eta}$ for $\eta\in\Z$.

	The following proposition gives the existence of a higher $(\eta,\al)$-skew Poisson derivation on a Poisson-Ore extension which is the semiclassical limit of an Ore extension satisfying suitable conditions. 
	
	Let $\mathcal{A}$ be a $\K[t^{\pm1}]$-algebra, let $\sigma$ be an automorphism of the $\K[t^{\pm1}]$-algebra $\mathcal{A}$ and let $\Delta$ be a $\K[t^{\pm1}]$-linear $\sigma$-derivation of $\mathcal{A}$. Recall that the Ore extension $\mathcal{R}=\mathcal{A}[x;\sigma,\Delta]$ is just the skew polynomial ring whose multiplication is defined by:
$$xa = \sigma(a) x + \Delta(a)$$ 
for all $a \in \mathcal{A}$.	
	
\begin{prop}
\label{SML}
Let $\mathcal{A}$ be a torsion free $\K[t^{\pm1}]$-algebra. Consider the Ore extension $\mathcal{R}=\mathcal{A}[x;\sigma,\Delta]$ and suppose that $R:=\mathcal{R}/(t-1)\mathcal{R}$ is a commutative $\K$-algebra. Then
\begin{enumerate}
	\item $R$ is a Poisson-Ore extension of the form $A[X;\al,\de]_P$, where $A=\mathcal{A}/(t-1)\mathcal{A}$, $X=\ov{x}$, $\al\in\D_P(A)$ and $\de$ is a Poisson $\al$-derivation of $A$. More precisely, we have 
	\[\al:=\frac{\sigma-\id}{t-1}\Big|_{t=1} \mbox{ and } \de:=\frac{\Delta}{t-1}\Big|_{t=1},\]
meaning that for all $a\in\mathcal{A}$ we have $\al(\ov{a})=\frac{\sigma(a)-a}{t-1}|_{t=1}$ and $\de(\ov{a}) =\frac{\Delta(a)}{t-1}|_{t=1}$.
	\item Suppose furthermore that $\Delta\sigma=t^{\eta}\sigma\Delta$ for some integer $\eta\in\K^{\times}$ and that
	\[\Delta^i(\mathcal{A})\subseteq (t-1)^i(i)!_{t^{\eta}} \mathcal{A}\]
for all $i\geq0$. Then $\de$ extends to an iterative, higher $(\eta,\al)$-skew Poisson derivation $(D_i)$ on $A$, which is locally nilpotent if $\Delta$ is locally nilpotent. More precisely, $D_i$ is defined by 
\[D_i(\ov{a})=\Big(\frac{\Delta^i(a)}{(t-1)^i(i)!_{t^{\eta}}}\Big)\Big|_{t=1}\]
for all $a\in\mathcal{A}$.
\end{enumerate}
\end{prop}

\begin{proof}
	1. First note that $(t-1)\mathcal{R}=(t-1)\mathcal{A}[x;\sigma,\Delta]$, where $(t-1)\mathcal{A}$ is a $(\sigma,\Delta)$-stable ideal of $\mathcal{A}$. So the corresponding quotient algebra is of the form 
	$$R=\mathcal{R}/(t-1)\mathcal{R}=\left( \mathcal{A}/(t-1)\mathcal{A} \right)[X]=A[X].$$
	We already know that $R$ is a Poisson algebra, so it just remains to prove that $R$ is a Poisson-Ore extension. Since $R$ is commutative, for $a\in\mathcal{A}$, we have 
\begin{align*}
0=\ov{x}\ov{a}-\ov{a}\ov{x}=\ov{(\sigma(a)-a)x+\Delta(a)}=\ov{(\sigma(a)-a)}X+\ov{\Delta(a)}.
\end{align*}
So $(\sigma(a)-a)\in(t-1)\mathcal{A}$ and $\Delta(a)\in(t-1)\mathcal{A}$ for all $a\in \mathcal{A}$. The Poisson bracket between $\ov{a}\in A$ and $X$ is given by:
\begin{align*}
	\{X,\ov{a}\}=\frac{\sigma(a)-a}{t-1}\Big|_{t=1}\ X+\frac{\Delta(a)}{t-1}\Big|_{t=1}.
\end{align*}
We set 
	\[\al:=\frac{\sigma-\id}{t-1}\Big|_{t=1} \mbox{ and } \de:=\frac{\Delta}{t-1}\Big|_{t=1}.\]
	One can easily check that $\al$ and $\de$ are well defined, and that $\al\in\D_P(A)$ and $\de$ is a Poisson $\al$-derivation on $A$. Thus 
	\[\{X,\ov{a}\}=\al(\ov{a})X+\de(\ov{a})\]
for all $a\in\mathcal{A}$, and the algebra $R$ is a Poisson-Ore extension of the form $A[X;\al,\de]_P$.

	2. We claim that one defines an iterative, higher $(\eta,\al)$-skew Poisson derivation $(D_i)$ on $A$ by:
\[D_i(\ov{a})=\Big(\frac{\Delta^i(a)}{(t-1)^i(i)!_{t^{\eta}}}\Big)\Big|_{t=1}\]
for all $i\geq0$ and all $a\in\mathcal{A}$. First, since $\Delta^i(\mathcal{A})\subseteq (t-1)^i(i)!_{t^{\eta}} \mathcal{A}$, it is straightforward to see that the map $D_i$ is well defined for all $i\geq0$.
	It remains to check that $(D_i)$ satisfies all the relevant axioms of Definition \ref{hd}. Axiom (A1) follows from the fact that $\ov{\sigma(a)}=\ov{a}$ for all $a\in\mathcal{A}$. Set $d_i=\frac{\Delta^i}{(i)!_{t^{\eta}}}$ 
for all $i\geq0$. Then (A3) follows easily from the identities
	\[d_i(\sigma-\id)=t^{i\eta}(\sigma-\id)d_i+(t^{i\eta}-1)d_i\]
for all $i\geq0$. The higher derivation $(D_i)$ is iterative since $d_id_j =  \binom{i+j}{j}_{t^{\eta}} d_{i+j}$. Moreover, it is clear that $(D_i)$ is locally nilpotent if $\Delta$ is.

	The verification of (A2) involves more computations, so the details are given here. Let $u,v\in\mathcal{A}$. Then 
	one can easily check that 
	\begin{align*}
d_n (uv)=\sum\limits_{i=0}^{n}\sigma^{n-i}d_i(u)d_{n-i}(v),
\end{align*}
so that for all $a,b \in \mathcal{A}$ we have	
	\begin{align*}
d_n\Big(\frac{[a,b]}{t-1}\Big)
&=\frac{1}{(t-1)}\Big(\sum\limits_{i=0}^{n-1}\sigma^{n-i}d_i(a)d_{n-i}(b)-\sum\limits_{i=0}^{n-1}\sigma^{n-i}d_i(b)d_{n-i}(a)+d_n(a)b-d_n(b)a\Big).
\end{align*}
Observe that for $i<n$:
\begin{align*}
	\sigma^{n-i}d_i(a)d_{n-i}(b)=\sum_{j=1}^{n-i}\sigma^{n-i-j}(\sigma-\id)d_i(a)d_{n-i}(b)+d_i(a)d_{n-i}(b).
\end{align*}
Thus
\begin{align*}	
d_n\Big(\frac{[a,b]}{t-1}\Big)
&=\frac{1}{(t-1)}\sum\limits_{i=0}^{n-1}\Big(\sum_{j=1}^{n-i}\sigma^{n-i-j}(\sigma-\id)d_i(a)d_{n-i}(b)-\sum_{j=1}^{n-i}\sigma^{n-i-j}(\sigma-\id)d_i(b)d_{n-i}(a)\Big)\\
		&\qquad\qquad\qquad\qquad+\sum\limits_{i=0}^{n}\frac{[d_i(a),d_{n-i}(b)]}{(t-1)}.\\	
\end{align*}
Dividing by $(t-1)^n$, and then projecting onto $R$, we get:
\begin{align*}
D_n(\{\ov{a},\ov{b}\})
=\sum\limits_{i=0}^n\{D_i(\ov{a}),D_{n-i}(\ov{b})\}+\sum\limits_{i=1}^ni\big(\al D_{n-i}(\ov{a})D_i(\ov{b})-\al D_{n-i}(\ov{b})D_i(\ov{a})\big).
\end{align*}
This proves (A2).
\end{proof}

	We can now state the main result of this section.

\begin{thm}
\label{fin}
Let $\mathcal{R}=\K[t^{\pm1}][x_1][x_2;\sigma_2,\Delta_2]\cdots[x_n;\sigma_n,\Delta_n]$ be an iterated Ore extension over $\K[t^{\pm1}]$, and denote by $\mathcal{R}_j$ the subalgebra $\K[t^{\pm1}][x_1][x_2;\sigma_2,\Delta_2]\cdots[x_j;\sigma_j,\Delta_j]$ for $1 \leq j \leq n$. We make the following assumptions:\\
{\rm (H1)} The torus $H=(\K^{\times})^r$ is acting rationally by $\K[t^{\pm1}]$-algebra automorphisms on $\mathcal{R}$ such that for all $i\in \{1, \dots ,n\}$
	
	$\bullet$ the indeterminate $x_i$ is an $H$-eigenvector with associated character $\ul{f}_i$;
	
and	
	
	$\bullet$ there exists $\ul{\gam}_i\in\Z^r$ such that $\eta_i:=-(\ul{\gam}_i|\ul{f}_i)\in\K^{\times}$;\\
{\rm (H2)} For all $2\leq i\leq n$, we have $\Delta_i\sigma_i=t^{\eta_i}\sigma_i\Delta_i$;\\
{\rm (H3)} For all $2\leq i\leq n$ and $k\geq0$, we have $\Delta_i^k(\mathcal{R}_{i-1})\subseteq(t-1)^k(k)!_{t^{\eta_i}}\mathcal{R}_{i-1}$;\\
{\rm (H4)} The automorphisms $\sigma_i$ satisfy $\sigma_i(x_j)=t^{\la_{ij}}x_j$ for $1\leq j<i\leq n$, where $\la_{ij}:=(\ul{\gam}_i|\ul{f}_j)$.

	Assume that $R:=\mathcal{R}/(t-1)\mathcal{R}$ is commutative. Then, for any $H$-invariant Poisson prime ideal $P$ of $R$, the field $\F R/P$ is Poisson isomorphic to a Poisson affine field.
\end{thm}

\begin{proof}

	We only need to check that $R$ satisfies all hypotheses of Theorem \ref{all hypo}.
	
	$\bullet$ First, we show that $R$ is an iterated Poisson-Ore extension of the form
	\[R=\K[X_1][X_2;\al_2,\de_2]_P\cdots[X_n;\al_n,\de_n]_P,\]
where each $\de_i$ extends to an iterative higher $(\eta_i,\al_i)$-skew Poisson derivation $(D_{i,k})_{k=0}^{\infty}$ on $R_{i-1}:=\K[X_1][X_2;\al_2,\de_2]_P\cdots[X_{i-1};\al_{i-1},\de_{i-1}]_P$. This result is proved by induction on $n$ using Proposition \ref{SML}. The case $n=1$ is trivial. 

	For $1\leq i\leq n-1$, assume that $R_{i}=\K[X_1][X_2;\al_2,\de_2]_P\cdots[X_{i};\al_{i},\de_{i}]_P$. Then we have 
\begin{align*}
	R_{n}=\frac{\mathcal{R}_{n}}{(t-1)\mathcal{R}_{n}}&=\frac{\mathcal{R}_{n-1}}{(t-1)\mathcal{R}_{n-1}}[X_{n};\al_{n},\de_{n}]_P\\
	&=\K[X_1][X_2;\al_2,\de_2]_P\cdots[X_{n};\al_{n},\de_{n}]_P,
\end{align*}
since $(t-1)\mathcal{R}_{n-1}$ is a $(\sigma_{n},\Delta_{n})$-stable ideal of $\mathcal{R}_{n-1}$. Note that 
	\[\al_n(X_j)=\frac{\sigma_n(x_j)-x_j}{t-1}\Big|_{t=1}=\frac{t^{\la_{nj}}-1}{t-1}x_j\Big|_{t=1}=\la_{nj}X_j,\]
for all $1\leq j\leq n$.

	Hypotheses (H2) and (H3) ensure that Assertion 2 of Proposition \ref{SML} applies, so $\de_{n}$ extends to an iterative higher $(\eta_{n},\al_{n})$-skew Poisson derivation $(D_{n,k})_{k=0}^{\infty}$ on $R_{n-1}$. It follows from Proposition \ref{SML} (and the induction hypothesis) that for $2\leq j\leq n$ and $k\geq0$ we have
	\[D_{j,k}:=\frac{\Delta_j^k}{(t-1)^k(k)!_{t^{\eta_j}}}\Big|_{t=1}.\]
	
$\bullet$	The next step is to show that for $2\leq j<i\leq n$ and $k\geq0$, we have the relations $\al_iD_{j,k}=D_{j,k}\al_i+k\la_{ij}D_{j,k}$. 
		
First we show by induction (on $k$) the following identities:
\begin{eqnarray} 
\label{formuleDelta}
\sigma_i\Delta_j^k=t^{k\la_{ij}}\Delta_j^k\sigma_i,
\end{eqnarray}
for $2\leq j<i\leq n$. If $k=1$ and $1\leq l<j$, then we have
\begin{align*}
	\sigma_i(x_jx_l)=\sigma_i\big(\sigma_j(x_l)x_j+\Delta_j(x_l)\big)=t^{\la_{ij}+\la_{il}+\la_{jl}}x_ix_j+\sigma_i\Delta_j(x_l),
\end{align*}
and
\begin{align*}
	\sigma_i(x_j)\sigma_i(x_l)=t^{\la_{ij}+\la_{il}}\big(\sigma_j(x_l)+\Delta_j(x_l)\big)=t^{\la_{ij}+\la_{il}+\la_{jl}}x_ix_j+t^{\la_{ij}}\Delta_j\sigma_i(x_l).
\end{align*}
So $\sigma_i\Delta_j(x_l)=t^{\la_{ij}}\Delta_j\sigma_i(x_l)$ for all $1\leq l<j<i\leq n$, as desired. Assume the result proved at rank $k$. Then we have
\begin{align*}
	\sigma_i\Delta_j^{k+1}=(\sigma_i\Delta_j)\Delta_j^k=t^{\la_{ij}}\Delta_j\sigma_i\Delta_j^k=t^{(k+1)\la_{ij}}\Delta_j^{k+1}\sigma_i,
\end{align*}
and (\ref{formuleDelta}) is proved.

Now it follows from (\ref{formuleDelta}) that
	\[(\sigma_i-\id)\Delta_j^k=t^{k\la_{ij}}\Delta_j^k(\sigma_i-\id)+(t^{k\la_{ij}}-1)\Delta_j^k.\]
Next, dividing both sides of this equation by $(t-1)^{k+1}(k)!_{t^{\eta_j}}$ , and then projecting on $R_{j-1}$, we obtain:
	\[\al_iD_{j,k}=D_{j,k}\al_i+k\la_{ij}D_{j,k}.\]

$\bullet$	Then we show that the torus $H$ is acting rationally by Poisson automorphisms on $R$. Since $(t-1)\mathcal{R}$ is $H$-invariant, we can consider the induced action of $H$ on the quotient algebra $R$. This is a rational action by automorphisms. Moreover this action respects the Poisson bracket of $R$. Indeed for $f,g\in\mathcal{R}$, by setting $F=\ov{f}$ and $G=\ov{g}$, we have:
\begin{align*}
	h(\{F,G\})&=h\Big(\Big(\frac{[f,g]}{t-1}\Big)\Big|_{t=1}\Big)=\Big(h\Big(\frac{[f,g]}{t-1}\Big)\Big)\Big|_{t=1}\\
	&=\Big(\frac{[h(f),h(g)]}{t-1}\Big)\Big|_{t=1}=\{h(F),h(G)\}
\end{align*}
for all $h\in H$.
	
$\bullet$	Fix $h\in H$ and set $h(x_j)=\mu_jx_j$, where $\mu_j\in\K^{\times}$ for all $1\leq j\leq n$. We are now going to show that
	\[h\big(D_{i,k}(X_j)\big)=\mu_i^kD_{i,k}\big(h(X_j)\big)\]
for all $1\leq j<i\leq n$ and all $k\geq0$.

	We start by observing that, for $k\geq1$ and $1\leq j<i\leq n$, we have:
	\[x_i\Delta_i^{k-1}(x_j)=\sigma_i(\Delta_i^{k-1}(x_j))x_i+\Delta_i^k(x_j).\]
Thus
\begin{align*}
	\Delta_i^k(x_j)&=x_i\Delta_i^{k-1}(x_j)-\sigma_i(\Delta_i^{k-1}(x_j))x_i\\
	&=x_i\Delta_i^{k-1}(x_j)-t^{\eta_i(1-k)+\la_{ij}}\Delta_i^{k-1}(x_j)x_i.
\end{align*}
Then it follows from an easy induction (on $k$) that for all $h\in H$ and $k \geq 0$ we have 
\begin{eqnarray}
\label{eq:hDelta}
h(\Delta_i^k(x_j))=\mu_j\mu_i^k\Delta_i^k(x_j).
\end{eqnarray}
Indeed, when $k=1$, we have 
	\[h(\Delta_i(x_j))=h(x_ix_j-t^{\la_{ij}}x_jx_i)=\mu_i\mu_j\Delta_i(x_j).\]
Next, assuming the result proved at rank $(k-1)$ we get:
\begin{align*}
	h(\Delta_i^k(x_j))&=h(x_i\Delta_i^{k-1}(x_j)-t^{\eta_i(1-k)+\la_{ij}}\Delta_i^{k-1}(x_j)x_i)\\
	&=\mu_i x_i\mu_j\mu_i^{k-1}\Delta_i^{k-1}(x_j)-t^{\eta_i(1-k)+\la_{ij}}\mu_j\mu_i^{k-1}\Delta_i^{k-1}(x_j)\mu_ix_i\\
	&=\mu_j\mu_i^k\Delta_i^k(x_j),
\end{align*}
as desired. As $D_{j,k}:=\frac{\Delta_j^k}{(t-1)^k(k)!_{t^{\eta_j}}}\Big|_{t=1}$, we deduce from (\ref{eq:hDelta}) that 
	\[h(D_{i,k}(X_j))=\mu_i^kD_{i,k}\big(h(X_j)\big)\]
	for all $k\geq0$ and for all $1 \leq j<i\leq n$, as required.

$\bullet$	 We conclude by noting that Hypothesis \ref{hyp} is clearly satisfied with $\rho_i=-\eta_i=(\ul{\gam}_i|\ul{f}_i)$ for all $1\leq i\leq n$ since $X_i$ is an $H$-eigenvector with associated character $\ul{f}_i$ for all $1\leq i\leq n$. 

	Hence all hypothesis of Theorem \ref{all hypo} are satisfied and so for any $H$-invariant Poisson prime ideal $P$ of $R$, the field $\F R/P$ is Poisson isomorphic to a Poisson affine field.
\end{proof}

When dealing with examples, the following lemma allows us to check Hypothesis (H2) of Theorem \ref{fin} only on the generators of the algebra under consideration.

\begin{lem}
\label{generators}
	Let $\mathcal{A}$ be a finitely generated $\K[t^{\pm1}]$-algebra and form the Ore extension $\mathcal{R}=\mathcal{A}[x;\sigma,\Delta]$ with $\Delta\sigma=t^{\eta}\sigma\Delta$ for an integer $\eta\in\K^{\times}$. Let $\{a_1,\dots,a_n\}$ be a set of generators of $\mathcal{A}$. If the conditions $\Delta^i(a_k)\in (t-1)^i(i)!_{t^{\eta}}\mathcal{A}$ are satisfied for all $k\in\{1,\dots,n\}$ and $i \geq 0$, then	\[\Delta^i(\mathcal{A})\subseteq(t-1)^i(i)!_{t^{\eta}}\mathcal{A}.\]
\end{lem}

\begin{proof}
Easy induction using the generalised quantum Leibniz formula:
	\[\Delta^i(ab)=\sum_{k=0}^i\binom{i}{k}_{t^{\eta}}\sigma^{i-k}\Delta^k(a)\Delta^{i-k}(b)\]
for $a,b\in\mathcal{A}$.
\end{proof}

In \cite[Section 5]{HH}, many iterated Ore extensions are described, and it is shown that lots of them actually satisfy the hypotheses of Theorem \ref{fin}. As a consequence, their semiclassical limits and their quotients by $H$-invariant Poisson prime ideals satisfy the quadratic Poisson Gel'fand-Kirillov problem. This includes (but is not limited to) the semiclassical limits of:
\begin{itemize}
	\item single parameter coordinate rings of odd-dimensional quantum Euclidean spaces;
	\item single parameter coordinate rings of quantum matrices;
	\item single parameter coordinate rings of even-dimensional quantum Euclidean spaces;
	\item single parameter coordinate rings of quantum symplectic spaces.
\end{itemize}
In the next section, we provide a detailed study of one of these families. More precisely, we will focus on the semiclassical limit of the coordinate rings of quantum matrices, and then on their quotients by  determinantal ideals.

\subsection{Semiclassical limit of the coordinate ring of $n\times n$ quantum matrices}

	The single parameter coordinate ring of quantum matrices $\mathcal{A}=\mathcal{O}_t\big(M_n(\K[t^{\pm1}])\big)$ is the $\K[t^{\pm1}]$-algebra given by $n^2$ generators $x_{11}, x_{12}, \dots,x_{nn}$ and relations
\begin{align*}
	x_{lm}x_{ij}=\left\{
    \begin{array}{llll}
        t^{-1}x_{ij}x_{lm} \qquad &l>i,\ m=j\\
         t^{-1}x_{ij}x_{lm} \qquad &l=i,\ m>j\\
         x_{ij}x_{lm} \qquad &l>i,\ m<j\\
         x_{ij}x_{lm}-(t-t^{-1})x_{im}x_{lj} \qquad &l>i,\ m>j.
    \end{array}\right.
\end{align*}
This algebra can also be presented as an iterated Ore extension over $\K[t^{\pm1}]$:
	\[\mathcal{O}_t\big(M_n(\K[t^{\pm1}])\big)=\K[t^{\pm1}][x_{11}][x_{12};\sigma_{12},\Delta_{12}]\cdots[x_{nn};\sigma_{nn},\Delta_{nn}],\]
where $\sigma_{lm}$ is the $\K[t^{\pm1}]$-automorphism of the appropriate subalgebra of $\mathcal{O}_t\big(M_n(\K[t^{\pm1}])\big)$ defined by 
\[\sigma_{lm}(x_{ij})= \left\{
			\begin{array}{lll}
				t^{-1}x_{ij} \qquad \qquad & \mbox{if } l>i \mbox{ and } m=j\\
				t^{-1}x_{ij}               & \mbox{if } l=i \mbox{ and } m>j\\
				x_{ij}                     & \mbox{if } l>i \mbox{ and } m\neq j,
			\end{array}
\right. 
\]
for all $(i,j)<_{\mbox{lex}}(l,m)$, and where $\Delta_{lm}$ is the $\K[t^{\pm1}]$-linear $\sigma_{lm}$-derivation such that 	
\[\Delta_{lm}(x_{ij})=\left\{
			\begin{array}{ll}
				-(t-t^{-1})x_{im}x_{lj} \ \ & \mbox{if } l>i \mbox{ and } m>j\\
				0 & \mbox{otherwise}
			\end{array}
\right.
\]
for all $(i,j)<_{\mbox{lex}}(l,m)$.

Observe that the torus $H=(\K^{\times})^{2n}$ acts rationally on $\mathcal{A}$ by automorphisms via:
	\[h(t)=t\ \mbox{ and }\ h(x_{ij})=h_ih_{n+j}x_{ij}\]
for all $1\leq i,j\leq n$. So $x_{ij}$ is an $H$-eigenvector with associated character 
	\[\ul{f}_{ij}=(0,\dots,0,1,0,\dots,0,1,0,\dots,0)\in\Z^{2n},\]
where the $1$s occur in $i$-th and $(n+j)$-th positions. For $1\leq l,m\leq n$, we define 
	\[\ul{\gam}_{lm}:=(1\dots,1,0,-1,\dots,-1,-2,-1,\dots,-1)\in\Z^{2n},\]
where the $0$ occurs in $l$-th position and the $(-2)$ in $(n+m)$-th position. We have $(\ul{\gam}_{lm}|\ul{f}_{lm})=-2$ for all $1\leq l,m\leq n$. To summarise, if $\car(\K)\neq2$, Hypothesis (H1) of Theorem \ref{fin} is satisfied. For $(i,j)<_{\mbox{lex}}(l,m)$ we have:
\begin{align*}
	(\ul{\gam}_{lm}|\ul{f}_{ij})=\left\{
					\begin{array}{ll}
					  -1\ \  &\mbox{if } l>i \mbox{ and } m=j\\
					  -1     &\mbox{if } l=i \mbox{ and } m>j\\
						0      &\mbox{if } l>i \mbox{ and } m\neq j.						
					\end{array}
		\right.
\end{align*}
Note that for all $(i,j)<_{\mbox{lex}}(l,m)$ we have $\sigma_{lm}(x_{ij})=t^{(\ul{\gam}_{lm}|\ul{f}_{ij})}x_{ij}$. Thus Hypothesis (H4) of Theorem \ref{fin} is satisfied.

	One can easily check that $\Delta_{lm}\sigma_{lm}=t^2\sigma_{lm}\Delta_{lm}$ for all $l,m$. Thus, Hypothesis (H2) of Theorem \ref{fin} is satisfied. Let $\mathcal{A}_{lm}$ be the subalgebra of $\mathcal{A}$ generated over $\K[t^{\pm1}]$ by $x_{11},x_{12}, \dots,x_{l,m-1}$. Note that  $\Delta_{lm}^k(x_{ij})=0$ for all $k\geq 2$ and
	\[\Delta_{lm}(x_{ij})= \left\{
			\begin{array}{ll}
				-(t-1)(t^{-1}+1)x_{im}x_{lj} \ \ & \mbox{if } l>i \mbox{ and } m>j\\
				0 & \mbox{otherwise.}
			\end{array}
		\right.\]
So we have $\Delta_{lm}^k(x_{ij})\in(t-1)^k(k)!_{t^2}\mathcal{A}_{lm}$ for all $(i,j)<_{\mbox{lex}}(l,m)$ and all $k\geq0$, and Hypothesis (H3) of Theorem \ref{fin} is satisfied thanks to Lemma \ref{generators}. So, if $\car\K\neq 2$, then we can apply Theorem \ref{fin} to $\mathcal{A}$. 

 Let $A=\mathcal{O}\big(M_n(\K)\big)=\mathcal{A}/(t-1)\mathcal{A}=\K[X_{11},\dots,X_{nn}]$ be the semiclassical limit of $\mathcal{A}$, where $X_{ij}=x_{ij}+(t-1)\mathcal{A}$. For  $(i,j)<_{\mbox{lex}}(l,m)$, the Poisson bracket on $A$ is given by:
	
	\[\{X_{lm},X_{ij}\}=\left\{
				\begin{array}{ll}
				-X_{ij}X_{lm}\ \ &\mbox{if } l>i \mbox{ and } m=j\\
				-X_{ij}X_{lm}    &\mbox{if } l=i \mbox{ and } m>j\\
				0                &\mbox{if } l>i \mbox{ and } m<j\\
				-2X_{im}X_{lj}    &\mbox{if } l>i \mbox{ and } m>j.
				\end{array}
	\right.
\] 
We deduce from the above discussion the following result.

\begin{thm}
\label{max}
Assume that $\car\K \neq 2$. Let $P$ be an $H$-invariant Poisson prime ideal of $A=\mathcal{O}\big( M_n(\K)\big)$. The field of fractions of $A/P$ is Poisson isomorphic to a Poisson affine field $\K_{\bo\mu}(Y_1,\dots,Y_m)$, where $m\leq n^2$ and $\bo\mu\in M_m(\K)$ is a skew-symmetric matrix.
\end{thm}

Note that when $\car\K=2$, our methods do not apply to $A$. However in this case $A$ is already a Poisson affine space and the quadratic Poisson Gel'fand-Kirillov problem is trivial.

\subsection{Quotients by Determinantal ideals}
	
	Assume that $\car\K \neq 2$. Determinantal ideals are ideals of $A=\mathcal{O}\big(M_n(\K)\big)$ generated by minors of a given size. More precisely, let $I$ and $J$ be subsets of $\{1,\dots,n\}$ with $|I|=|J|$. We denote by $[I|J]$ the determinant
	\[ [I|J]:=\det\Big((X_{ij})_{(i,j)\in I\times J}\Big).\]
	Such a determinant is called a \textit{minor of size $|I|$}. For all $k\in\{0,\dots,n-1\}$, the \textit{determinantal ideal} $\mathcal{P}_k$ is the ideal generated by all $(k+1)\times(k+1)$ minors of $A$. Note that $\mathcal{P}_k$ contains all minors of size bigger than $k+1$ by Laplace Expansion.
	
	 Fix $0\leq k\leq n-1$. We claim that the Poisson field $\F(A/\mathcal{P}_k)$ is Poisson isomorphic to a Poisson affine field. For, we just need to show that $\mathcal{P}_k$ is an $H$-invariant Poisson prime ideal by Theorem \ref{max}. First, it is well known that $\mathcal{P}_k$ is a prime ideal, see for instance \cite[Theorem 6.3]{detring}. Moreover, $\mathcal{P}_k$ is clearly $H$-invariant, so to apply Theorem \ref{max} to $A/\mathcal{P}_k$, it only remains to prove that $\mathcal{P}_k$ is a Poisson ideal. It is probably well known, but we have not been able to find the statement in the literature. The following lemma (re-)establishes this result.
	
\begin{lem}
\label{Poisson}
For all $k\in\{0,\dots,n-1\}$, the ideal $\mathcal{P}_k$ is a Poisson ideal of $A$. 
\end{lem}

\begin{proof}
	Note that any minor of $A$ is the coset of a so-called \textit{quantum minor} of $\mathcal{A}$. See \cite[Introduction]{GoLe} for more details about quantum minors. In \cite[Lemma \textbf{5.1}]{GoLe} the authors give commutation relations between quantum minors and generators of $\mathcal{A}$ which easily lead (by semiclassical limit) to the following Poisson brackets between minors and generators of $A$. Let $r,c\in\{1,\dots,n\}$ and $I,J\subseteq\{1,\dots,n\}$ with $|I|=|J|\geq1$. For $1\leq i<j\leq n$, we define $[i,j]:=\{i,i+1,\dots,j\}$.
\begin{itemize}
	\item If $r\in I$ and $c\in J$, then $\big\{X_{rc},[I|J]\big\}=0$.
	\item If $r\in I$ and $c\notin J$, then
	\[\big\{X_{rc},[I|J]\big\}=-[I|J]X_{rc}-2\sum_{j\in J,j>c}(-1)^{-|J\cap[c,j]|}[I|J\cup\{c\}\setminus\{j\}]X_{rj}.\]
	\item If $r\notin I$ and $c\in J$, then
	\[\big\{X_{rc},[I|J]\big\}=[I|J]X_{rc}+2\sum_{i\in I,i<r}(-1)^{-|I\cap[i,r]|}[I\cup\{r\}\setminus\{i\}|J]X_{ic}.\]
	\item If $r\notin I$ and $c\notin J$, then
\begin{align*}
	\big\{X_{rc},[I|J]\big\}&=2\sum_{i\in I,i<r}(-1)^{-|I\cap[i,r]|}[I\cup\{r\}\setminus\{i\}|J]X_{ic}\\
									&-2\sum_{j\in J,j>c}(-1)^{-|J\cap[c,j]|}[I|J\cup\{c\}\setminus\{j\}]X_{rj}.
\end{align*}
\end{itemize}
Hence $\big\{X_{rc},[I|J]\big\}\in\mathcal{P}_k$ for all $[I|J]\in\mathcal{P}_k$ and all $1\leq r,c\leq n$.
\end{proof}

We are ready to conclude by the following result.

\begin{thm}
	Let $0\leq k\leq n-1$. The field of fractions $\F A/\mathcal{P}_k$ is Poisson isomorphic to a Poisson affine field $\K_{\bo\mu}(Y_1,\dots,Y_m)$, where $m\leq n^2$ and $\bo\mu\in M_m(\K)$ is a skew-symmetric matrix.
\end{thm}

\end{document}